\documentclass[11pt]{article}
\usepackage{fullpage,amsthm,amssymb,amsmath}
\usepackage{graphicx,algorithmic,algorithm}%
\usepackage{enumerate}
\usepackage{tikz}
\usetikzlibrary{shapes}
\usepackage{hyperref}
\usepackage[normalem]{ulem}
\hypersetup{
    pdftitle=   {Chi-boundedness of graph classes excluding two vertex-minors},
   pdfauthor=  {Hojin Choi and O-joung Kwon and Sang-il Oum and Paul Wollan}
}
\usepackage{authblk}
\usepackage{mathabx}
\newtheorem{theorem}{Theorem}[section]
\newtheorem{lemma}[theorem]{Lemma}
\newtheorem{proposition}[theorem]{Proposition}
\newtheorem{corollary}[theorem]{Corollary}
\newtheorem{conjecture}[theorem]{Conjecture}
\newtheorem{claim}[theorem]{Claim}
\newtheorem*{thm:main}{Theorem~\ref{thm:main}}
\newenvironment{clproof}{\begin{list}{}{%
              \setlength{\leftmargin}{5mm}%
              } \item {\it Proof.} }{\hfill$\lozenge$\end{list}\medskip}

\newcommand\abs[1]{\lvert #1\rvert}

\newcommand{\pivot}{\wedge}

\def\K_#1{{K_{#1}}}
\def\S_#1{\overline{K_{#1}}}

\begin{document}
\title{Chi-boundedness of graph classes excluding wheel vertex-minors}

\author[1]{Hojin Choi}

\affil[1]{Department of Mathematical Sciences, KAIST,  Daejeon, South Korea.}

\author[2]{O-joung Kwon\thanks{Previous affiliation : Logic and Semantics, Technische Universit\"at Berlin, Berlin, Germany.
Supported by the European Research Council (ERC) under the European Union's Horizon 2020 research and innovation programme (ERC consolidator grant DISTRUCT, agreement No. 648527).}}

\affil[2]{Department of Mathematics, Incheon National University, Incheon, South Korea.}

\author[1]{Sang-il Oum\thanks{Supported by the National Research Foundation of Korea (NRF) grant funded by the Korea government (MSIT) (No. NRF-2017R1A2B4005020).}}

\author[3]{Paul Wollan\thanks{Supported by the European Research Council (ERC) under the European Union's Seventh Framework Programme (FP7/2007-2013)/ERC Grant Agreement no. 279558.}}
\affil[3]{Department of Computer Science, University of Rome, ``La Sapienza'', Rome, Italy.}

\date\today
\maketitle

\footnotetext{E-mail addresses: \texttt{hjchoi0330@gmail.com} (H. Choi), \texttt{ojoungkwon@gmail.com} (O. Kwon), \texttt{sangil@kaist.edu} (S. Oum), \texttt{paul.wollan@gmail.com} (P. Wollan)}

\begin{abstract}
A class of graphs is \emph{$\chi$-bounded} if there exists a function
$f:\mathbb N\rightarrow \mathbb N$ such that for every graph $G$ in
the class and an induced subgraph $H$ of $G$, if $H$ has no clique of
size $q+1$, then the chromatic number of
$H$ is less than or equal to $f(q)$.
We denote by $W_n$ the wheel graph on $n+1$ vertices.
We show that the class of graphs having no vertex-minor isomorphic to $W_n$ is $\chi$-bounded. 
This generalizes several previous results; $\chi$-boundedness for circle graphs, for graphs having no $W_5$ vertex-minors, and for graphs having no fan vertex-minors.
\end{abstract}

\section{Introduction}

All graphs in this paper are simple and undirected. A \emph{clique} of
a graph is
a set of pairwise adjacent vertices. The \emph{clique
  number} of a graph $G$, denoted by $\omega(G)$, is the maximum
number of vertices in a clique in $G$. We denote the chromatic number of a graph $G$ by $\chi(G)$. 

Gy\'{a}rf\'{a}s~\cite{Gyarfas1987} introduced the concept of a $\chi$-bounded class of graphs. A class $\mathcal C$ of graphs is \emph{$\chi$-bounded} if there exists a function $f:\mathbb N\rightarrow \mathbb N$ such that for every graph $G\in \mathcal C$ and an induced subgraph $H$ of $G$, $\chi(H)\le f(\omega(H))$. 
Such a function $f$ is called a \emph{$\chi$-bounding function}. 
Gy\'{a}rf\'{a}s~\cite{Gyarfas1987} proved that for every positive integer $k$, the class of graphs with no induced path of length $k$ is $\chi$-bounded. 

A \emph{vertex-minor} of a graph $G$ is an induced subgraph of a graph
that can be obtained from $G$ by a sequence of \emph{local
  complementations} \cite{Bouchet1987a,Bouchet1987b,Bouchet1988,Bouchet1989a,Bouchet1990,Oum2004}.
The precise
definition will be presented in Section~\ref{sec:prelim}.

As graph minors are motivated by the study of planar graphs, 
one of the major motivations to study vertex-minors is due to its
close relation to circle graphs. 
\emph{Circle graphs} are intersection graphs of chords on a
circle. Vertex-minors of a circle graph are circle graphs, as local
complementations preserve the property of being circle graphs.

Gy\'{a}rf\'{a}s~\cite{Gyarfas19851, Gyarfas19852} proved the following theorem.
\begin{theorem}[Gy\'{a}rf\'{a}s~{\cite{Gyarfas19851, Gyarfas19852}}]\label{thm:circlegraph}
The class of circle graphs is $\chi$-bounded.
\end{theorem}

Dvo\v{r}\'{a}k and  Kr\'{a}l'~\cite{DvorakK2012} proved that
graphs of rank-width at most $k$ are also $\chi$-bounded and it is
also the case that the class of graphs of rank-width at most $k$ is
closed under taking vertex-minors. 

These motivate the following conjecture of Geelen (see~\cite{DvorakK2012}).
\begin{conjecture}[Geelen~{(see \cite{DvorakK2012})}]\label{conj:geelen}
For every graph $H$, the class of graphs with no $H$ vertex-minor is $\chi$-bounded.   
\end{conjecture}

Conjecture~\ref{conj:geelen} is known to be true for the following cases. Here, for an integer $k\ge 3$, a \emph{wheel graph} $W_k$ is a graph that consists of an induced cycle on $k$ vertices and an additional vertex adjacent to all vertices of the induced cycle. 
\begin{enumerate}[(I)]
\item Conjecture~\ref{conj:geelen} is true if $H$ is a vertex-minor of $W_5$,  as shown by Dvo\v{r}\'{a}k
  and Kr\'{a}l'~\cite{DvorakK2012}.  Bouchet~\cite{Bouchet1994} proved
  that a graph is a circle graph if and only if the graph has none of
  $W_5, W_7$, and $F_7$ as a vertex-minor, where $F_7$ is the (unique) $7$-vertex
  bipartite graph such that $F_7-v$ is a cycle of length 6 for some vertex
  $v$ of degree $3$.
  Geelen~\cite{geelen1995} gave a decomposition theorem of
  graphs with no $W_5$ vertex-minor, using circle graphs as one of the
  building blocks by applying a theorem of Bouchet. Dvo\v{r}\'{a}k and
  Kr\'{a}l'~\cite{DvorakK2012} used the decomposition theorem of Geelen and Theorem~\ref{thm:circlegraph} to prove that the class of graphs with no $W_5$ vertex-minor is $\chi$-bounded. 
\item Conjecture~\ref{conj:geelen} is true if $H$ is a vertex-minor of a fan graph (a fan graph is a graph
  obtained from the wheel graph by removing a vertex of degree $3$),
  as shown by I.~Choi, Kwon, and Oum~\cite{ChoiKO2016}. 

  This implies that Conjecture~\ref{conj:geelen} is true for all $H$ such that $H$ is a cycle, as every cycle is a vertex-minor of a sufficiently large fan graph.  For such $H$, the conjecture also follows from a recent theorem of Chudnovsky, Seymour, and
  Scott~\cite{ChudnovskySS2016}, proving a conjecture of Gy\'{a}rf\'{a}s that the class of graphs
  having no induced cycles of length at least $k$ is $\chi$-bounded
  for all $k$.
\end{enumerate}

We prove Conjecture~\ref{conj:geelen} for $H = W_k$ for all $k \ge 3$,  thus implying 
both (I) and (II).
\begin{thm:main}
For every integer $n\ge 3$, the class of graphs with no  $W_n$ vertex-minor is $\chi$-bounded.
\end{thm:main}

Our theorem also provides an alternative proof of  Theorem~\ref{thm:circlegraph}, as $W_n$ is not a circle graph for $n\ge 5$. 
Of course, (I) implies Theorem~\ref{thm:circlegraph}, but the proof of (I) by Dvo\v{r}\'{a}k and Kr\'{a}l'~\cite{DvorakK2012} depends on Theorem~\ref{thm:circlegraph}.
Moreover, (II) does not imply
Theorem~\ref{thm:circlegraph} since a fan graph is a circle graph.

The paper is organized as follows.  
Section~\ref{sec:prelim} provides some preliminary concepts. 
Section~\ref{sec:overview} gives a high level overview of the proof of our main theorem. 
 Section~\ref{sec:regpartlem} presents a lemma that will help us to arrange finite sets of reals. 
Section~\ref{sec:ramsey} proves a variant of Ramsey's theorem.
In Sections~\ref{sec:manufacturingwheel} and~\ref{sec:patchedpath}, we explain how to obtain a wheel graph from several large graphs as a vertex-minor.
We prove our main theorem in Section~\ref{sec:mainthm}.

\section{Preliminaries}\label{sec:prelim}

For a graph $G$, let $V(G)$ and $E(G)$ denote the vertex set and the edge set of $G$, respectively. 
Let $G$ be a graph.
For $S\subseteq V(G)$, we denote by $G[S]$ the subgraph of $G$ induced by $S$. 
For $v\in V(G)$ and $S\subseteq V(G)$, we denote by $G- v$ the graph obtained from $G$ by removing $v$, and by $G- S$ the graph obtained by removing all vertices in $S$. 
For $F\subseteq E(G)$, we denote by $G- F$ the graph with vertex set $V(G)$ and edge set $E(G) \setminus F$.  
For $v\in V(G)$, the set of \emph{neighbors} of $v$ in $G$ is denoted by $N_G(v)$, and the \emph{degree} of $v$ is the size of $N_G(v)$.
For $S\subseteq V(G)$, we denote by $N_G(S)$ the set of vertices in $V(G)\setminus S$ having a neighbor in $S$.
For an edge $e$ of a graph $G$, we denote by $G/e$ the graph obtained by contracting $e$. 
Note we are only considering simple graphs, so we delete any parallel edges which arise from contracting an edge.  A graph $H$ is a \emph{subdivision} of $G$ if $H$ can be obtained from $G$ by replacing each edge $vw$ by a path with at least one edge whose end vertices are $v$ and $w$.

For a vertex $v$ in a graph $G$, to perform \emph{local complementation} at $v$, 
replace the subgraph of $G$ induced on $N_G(v)$ 
by its complement graph. 
We write $G*v$ to denote the graph obtained from $G$ by applying
local complementation at $v$.
Two graphs $G$ and $H$ are \emph{locally equivalent} if $G$ can be obtained from $H$ by a sequence of local complementations.
A graph $H$ is a \emph{vertex-minor} of a graph $G$
if $H$ is an induced subgraph of a graph which is locally equivalent to $G$.

For an edge $uv$ of a graph $G$, to \emph{pivot} the edge $uv$ in $G$, denoted $G\pivot uv$, perform the series of local complementations $G*u*v*u$.
Note that $G\pivot uv$ is identical to the graph obtained from $G$ by flipping the adjacency relation between every pair of vertices $x$ and $y$ where $x$ and $y$ are contained in distinct sets of $N_G(u)\setminus (N_G(v)\cup \{v\})$, $N_G(v)\setminus (N_G(u)\cup \{v\})$, and $N_G(u)\cap N_G(v)$, and finally swapping the labels of $u$ and $v$.  To \emph{flip} the adjacency relation between two vertices, we delete the edge if it exists and add it otherwise. 

For a vertex $v$ of a graph $G$ with exactly two neighbors $v_1$ and $v_2$ that are non-adjacent, the series of operations $(G*v)-v$ is called \emph{smoothing} the vertex $v$. The resulting graph is equivalently the graph obtained by contracting an edge incident with $v$.
Note that if $H$ is a subdivision of $G$, then $G$ is a vertex-minor of $H$ because we can construct $G$ from $H$ by repeatedly smoothing vertices.

We describe another type of contraction that creates a graph isomorphic to  a vertex-minor of the original graph.
For a vertex set $S$ of a graph $G$ where $G[S]$ is connected, we denote by $G/S$ the graph obtained by contracting all edges in $G[S]$.
Thus, all vertices in $S$ are identified to one vertex in $G/S$. In general, $G/S$ is not a vertex-minor of $G$; the following lemma describes a situation, which will be useful in the coming arguments, where $G/S$ is isomorphic to a vertex-minor of $G$.

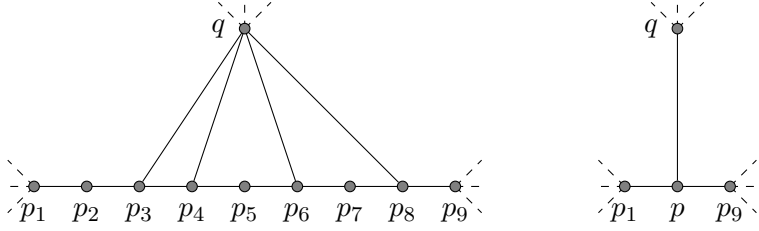
\begin{figure}
  \centering
  \begin{tikzpicture}[scale=0.7]
    \tikzstyle{w}=[circle,draw,fill=black!50,inner sep=0pt,minimum width=4pt]

   \draw (1,0)--(9,0);
   \draw[dashed] (5,3)--(5.5,3.5);
   \draw[dashed] (5,3)--(5,3.5);
   \draw[dashed] (5,3)--(4.5,3.5);

   \draw[dashed] (1,0)--(1-.5,.5);
   \draw[dashed] (1,0)--(1-.5,0);
   \draw[dashed] (1,0)--(1-.5,-.5);

   \draw[dashed] (9,0)--(9.5,.5);
   \draw[dashed] (9,0)--(9.5,0);
   \draw[dashed] (9,0)--(9.5,-.5);
   
 \foreach \y in {1,...,9}{
      \draw (\y,0) node [w] (a\y) {};
     \node at (\y, -.5) {$p_{\y}$};
     }
 \foreach \y in {3,4,6, 8}{
	\draw (5,3)--(a\y);
    }

      \draw (5,3) node [w] (x) {};
     \node at (4.5, 3) {$q$};

   \end{tikzpicture}\qquad\qquad
  \begin{tikzpicture}[scale=0.7]
    \tikzstyle{w}=[circle,draw,fill=black!50,inner sep=0pt,minimum width=4pt]

   \draw (4,0)--(6,0);
   \draw[dashed] (5,3)--(5.5,3.5);
   \draw[dashed] (5,3)--(5,3.5);
   \draw[dashed] (5,3)--(4.5,3.5);

   \draw[dashed] (4,0)--(4-.5,.5);
   \draw[dashed] (4,0)--(4-.5,0);
   \draw[dashed] (4,0)--(4-.5,-.5);

   \draw[dashed] (6,0)--(6.5,.5);
   \draw[dashed] (6,0)--(6.5,0);
   \draw[dashed] (6,0)--(6.5,-.5);
   
 \foreach \y in {4,...,6}{
      \draw (\y,0) node [w] (a\y) {};
     }
     \node at (4, -.5) {$p_1$};
     \node at (5, -.5) {$p$};
     \node at (6, -.5) {$p_9$};

 \foreach \y in {5}{
	\draw (5,3)--(a\y);
    }

      \draw (5,3) node [w] (x) {};
     \node at (4.5, 3) {$q$};

   \end{tikzpicture}
     \caption{The graph $G$ and a contraction $G/\{p_2, p_3, \ldots, p_8\}$.}\label{fig:contraction}
\end{figure}

\begin{lemma}\label{lem:fancontract}
Let $m\ge 4$ be an integer. Let $G$ be a graph and let $\{p_1, \ldots, p_m\}\cup \{q\}$ be a vertex set of $G$ such that
\begin{itemize}
\item $p_1p_2 \cdots p_m$ is an induced path in $G$, 
\item there are no edges between $\{p_2, \ldots, p_{m-1}\}$ and  $V(G)\setminus (\{p_1, \ldots, p_m\}\cup \{q\})$, 
\item $q$ has at least one neighbor in $\{p_3, \ldots, p_{m-1}\}$, and no neighbors in $\{p_1, p_2, p_m\}$. 
\end{itemize}
Then $G/\{p_2, p_3, \ldots, p_{m-1}\}$ is isomorphic to a vertex-minor of $G$. 
\end{lemma}
\begin{proof}
Let $G':=G/\{p_2, p_3, \ldots, p_{m-1}\}$ and let $p$ be the contracted vertex in $G'$.
We depict in Figure~\ref{fig:contraction}.
 We simulate this contraction as follows:
\begin{enumerate}
\item First if there is a vertex of degree $2$ in $\{p_3, \ldots, p_{m-1}\}$, then we smooth it. 
We may assume that there are no vertices of degree $2$ in $p_3, \ldots, p_{m-1}$.
\item If $m\ge 7$, then we apply local complementation at $p_4$ and remove it. This local complementation removes edges $qp_3$ and $qp_5$, and links $p_3$ and $p_5$.
Then we smooth $p_3$ and $p_5$. By applying this procedure repeatedly, we may assume  $m\in \{4,5,6\}$.
\item If $m=4$, then we smooth $p_2$. 
If $m=5$, then we apply local complementation at $p_3$ and remove it, and then smooth $p_4$.
If $m=6$, then  we pivot $p_3p_4$ and remove $p_3$ and $p_4$, and then smooth $p_5$.
\end{enumerate}
It is not difficult to see that each resulting graph is isomorphic to $G'$.
\end{proof}

\section{Overview of the approach}\label{sec:overview}

We begin by taking a leveling of the given graph.
A sequence $L_0, L_1, \ldots, L_m$ of disjoint subsets of the vertex set of a graph $G$ is called a \emph{leveling} in $G$ if 
\begin{enumerate}
\item $\abs{L_0} = 1$, and
\item for each $i\in\{1, \ldots, m\}$, every vertex in $L_i$ has a neighbor in $L_{i-1}$, and has no neighbors in $L_j$ for all $j\in\{0, \ldots, i-2\}$.
\end{enumerate}
Each $L_i$ is called a \emph{level}. 
We can obtain a leveling that covers all vertices in a connected graph by fixing a vertex $v$, and taking $L_i$ as the set of all vertices at distance $i$ from $v$.
Given a leveling $L_0, L_1, \dots, L_m$,
if each $L_i$ can be colored by $x$ colors, then 
the whole graph can be colored by $2x$ colors, because there are no edges between $L_{i}$ and $L_{j}$ for $\abs{j-i}\ge 2$.
Therefore, starting with a connected graph with sufficiently large chromatic number, we may assume that there is some level $L_i$ that still has large chromatic number.

\begin{figure}
  \centering
  \begin{tikzpicture}[scale=0.7]
    \tikzstyle{w}=[circle,draw,fill=black!50,inner sep=0pt,minimum width=4pt]

   \draw (1,0)--(8,0);
   \draw (2.7,1.3)--(9.7,1.3);
   \draw (1,0)--(2.7,1.3);\draw(8,0)--(9.7,1.3);

        \foreach \y in {1,...,8}{
      \draw (\y,0) node [w] (c\y) {};
      \draw (\y+1.7,1.3) node [w] (d\y) {};
    }
 
 \foreach \y in {2,...,7}{
      \draw (\y+1,3) node [w] (e\y) {};
      \draw(c\y)--(e\y)--(d\y);
      \draw (5.5,4.5)--(e\y);
    }

      \draw (5.5, 4.5) node [w] (x) {};
      \node at (5.5, 5) {$v$};

   \end{tikzpicture}
    \begin{tikzpicture}[scale=0.7]
    \tikzstyle{w}=[circle,draw,fill=black!50,inner sep=0pt,minimum width=4pt]

   \draw (1,0)--(8,0);
   \draw (2.7,1.3)--(9.7,1.3);
   \draw (1,0)--(2.7,1.3);\draw(8,0)--(9.7,1.3);

        \foreach \y in {1,...,8}{
      \draw (\y,0) node [w] (c\y) {};
      \draw (\y+1.7,1.3) node [w] (d\y) {};
    }
 
 \foreach \y in {2,...,7}{
      \draw(c\y)--(5.5, 4)--(d\y);
      \draw(c\y)--(d\y);
    }   
      \draw (5.5, 4) node [w] (x) {};
      \node at (5.5, 4.5) {$v$};

   \end{tikzpicture}

  \caption{A problematic case.}\label{fig:badcase}
\end{figure}
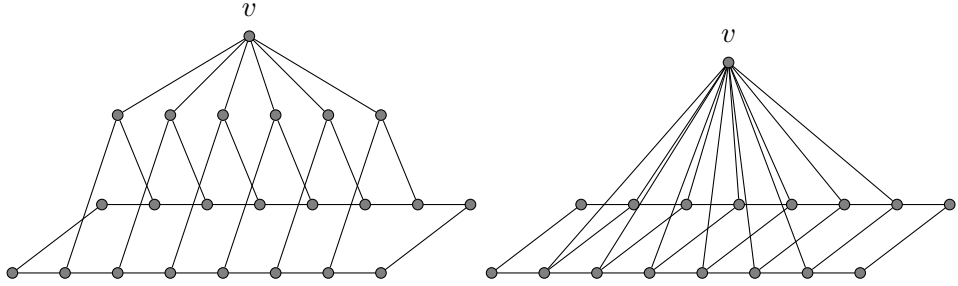

A natural approach to find a wheel vertex-minor is to find a long induced cycle in the level $L_i$ with large chromatic number, using the result by Chudnovsky, Scott, and Seymour (Theorem~\ref{thm:longcycle}), 
and then to construct a large wheel vertex-minor using the connected subgraph on $L_0\cup \cdots \cup L_{i-1}$. 
However, this strategy does not work well. 
For instance, we may find a graph depicted in the first figure of  Figure~\ref{fig:badcase}.
In this graph, if we apply local complementations to create edges from $v$ to the bottom cycle, 
then we obtain a graph obtained from a large wheel by adding some parallel chords, depicted in the right-hand figure.
At this point, it is difficult to remove these chords to finally obtain a wheel graph as a vertex-minor.

To avoid such problems, we aim to find a similar structure, but having two disjoint large independent sets having regular neighbors on the cycle. One simple example is depicted in the first figure of Figure~\ref{fig:badcasenewway}.
In this example, one independent set of vertices $w_i$ is used to create a vertex having many neighbors on the cycle, and the second set of vertices $z_j$ is used to remove the newly created chords. We depict this procedure in Figure~\ref{fig:badcasenewway}. 
Briefly speaking, to remove the chords that are newly created from $w_i$'s, we want to add new chords that does not share an end vertex with chords created from $w_i$'s, and then by pivoting these edges, we remove chords created from $w_i$'s. 
We need more involved arguments for dealing with general cases. This is one of the main procedures we will utilize to find a large wheel as a vertex-minor.

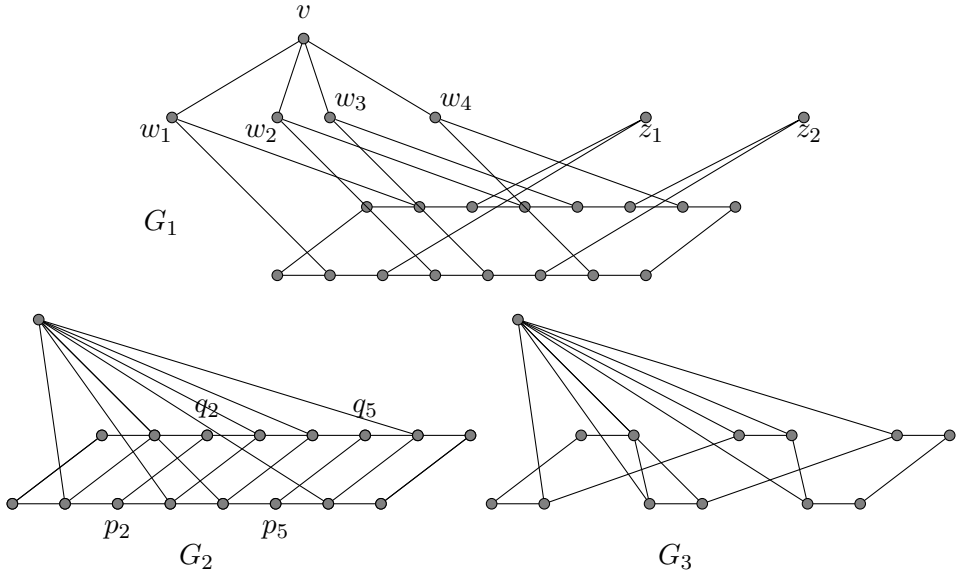
\begin{figure}
  \centering
  \begin{tikzpicture}[scale=0.7]
    \tikzstyle{w}=[circle,draw,fill=black!50,inner sep=0pt,minimum width=4pt]

   \draw (1,0)--(8,0);
   \draw (2.7,1.3)--(9.7,1.3);
   \draw (1,0)--(2.7,1.3);\draw(8,0)--(9.7,1.3);

    \foreach \y in {1,...,8}{
      \draw (\y,0) node [w] (c\y) {};
      \draw (\y+1.7,1.3) node [w] (d\y) {};
     }

 \foreach \y in {2,4,5,7}{
      \draw (\y-4+1,3) node [w] (e\y) {};
      \draw(c\y)--(e\y)--(d\y);
      \draw (1.5,4.5)--(e\y);
     }

 \foreach \y in {3,6}{
      \draw (\y+4+1,3) node [w] (f\y) {};
      \draw(c\y)--(f\y)--(d\y);
     }

      \draw (1.5, 4.5) node [w] (x) {};

      \node at (1+1.7-4, 2.7) {$w_{1}$};
      \node at (3+1.7-4, 2.7) {$w_{2}$};
      \node at (4.7+1.7-4, 3.3) {$w_{3}$};
      \node at (6.7+1.7-4, 3.3) {$w_{4}$};

      \node at (2+2.1+4, 2.7) {$z_{1}$};
      \node at (5+2.1+4, 2.7) {$z_{2}$};

      \node at (1.5, 5) {$v$};
      \node at (-1.2, 1) {$G_1$};

   \end{tikzpicture}\vskip 0.4cm
     \begin{tikzpicture}[scale=0.7]
    \tikzstyle{w}=[circle,draw,fill=black!50,inner sep=0pt,minimum width=4pt]

      \draw(1.5,3.5)--(2,0);
      \draw(1.5,3.5)--(4,0);
      \draw(1.5,3.5)--(5,0);
      \draw(1.5,3.5)--(7,0);
      \draw(1.5,3.5)--(3.7,1.3);
      \draw(1.5,3.5)--(5.7,1.3);
      \draw(1.5,3.5)--(6.7,1.3);
      \draw(1.5,3.5)--(8.7,1.3);
   
   \foreach \y in {1,...,8}{
      \draw (\y,0) node [w] (e\y) {};
      \draw (\y+1.7,1.3) node [w] (f\y) {};
     }

 \foreach \y in {1,...,8}{
      \draw (e\y)--(f\y);
     }

    \foreach \y in {2,5}{
      \node at (\y+1, -0.5) {$p_{\y}$};
      \node at (\y+2.7, 1.8) {$q_{\y}$};
    }

   \draw (1,0)--(8,0);
   \draw (2.7,1.3)--(9.7,1.3);
   \draw (1,0)--(2.7,1.3);\draw(8,0)--(9.7,1.3);

    \foreach \y in {1,...,8}{
      \draw (\y,0) node [w] (c\y) {};
      \draw (\y+1.7,1.3) node [w] (d\y) {};

     }
      \draw (1.5, 3.5) node [w] (x) {};
      \node at (4.5, -1) {$G_2$};

   \end{tikzpicture}
     \begin{tikzpicture}[scale=0.7]
    \tikzstyle{w}=[circle,draw,fill=black!50,inner sep=0pt,minimum width=4pt]

     \draw(1.5,3.5)--(2,0);
      \draw(1.5,3.5)--(4,0);
      \draw(1.5,3.5)--(5,0);
      \draw(1.5,3.5)--(7,0);
      \draw(1.5,3.5)--(3.7,1.3);
      \draw(1.5,3.5)--(5.7,1.3);
      \draw(1.5,3.5)--(6.7,1.3);
      \draw(1.5,3.5)--(8.7,1.3);
    \draw (1,0)--(2,0);\draw (4,0)--(5,0);\draw (7,0)--(8,0);
   \draw (2.7,1.3)--(3.7, 1.3);\draw (5.7, 1.3)--(6.7, 1.3); \draw(8.7, 1.3)--(9.7,1.3);
   \draw (1,0)--(2.7,1.3);\draw(8,0)--(9.7,1.3);

	\draw(2,0)--(5.7,1.3); \draw(6.7,1.3)--(7,0);
	\draw(3.7, 1.3)-- (4, 0); \draw (5,0)--(8.7, 1.3);

    \foreach \y in {1,2,4,5,7,8}{
      \draw (\y,0) node [w] (c\y) {};
      \draw (\y+1.7,1.3) node [w] (d\y) {};

     }
      \draw (1.5, 3.5) node [w] (x) {};
      \node at (4.5, -1) {$G_3$};

   \end{tikzpicture}

  \caption{An example procedure to find a large wheel. The graph $G_2$ is $(G_1*w_1*w_2*w_3*w_4*z_1*z_2)-\{w_1, w_2, w_3, w_4, z_1, z_2\}$, and 
  the graph $G_3$ is $(G_2\pivot p_2q_2\pivot p_5q_5)-\{p_2, p_5, q_2, q_5\}$. Since $G_3$ is isomorphic to a subdivision of $W_{8}$, it contains $W_{8}$ as a vertex-minor.}\label{fig:badcasenewway}
\end{figure}

Our argument begins with a structure arising from recursively taking repeated levelings. 
Explicitly, we aim to find pairwise disjoint vertex sets $X_i$ and $Y_1, \ldots, Y_i$ and $Z_1, \ldots, Z_i$ in a graph $G$ with sufficiently large $i$ such that   
\begin{itemize}
\item  $G[X_i]$ has large chromatic number,
\item for each vertex $v\in X_i$ and each $x\in \{1, \ldots, i\}$, $v$ has a neighbor in $Y_x$ and no neighbors in $Z_x$, 
\item for each $x\in \{1, \ldots, i\}$, every vertex in $Y_x$ has a neighbor in $Z_x$, 
\item for each $x\in \{1, \ldots, i\}$, there exists a vertex $r_x\in Z_x$ where 
for every $v\in N_G(Y_x)\cap Z_x$, there is a path $P$ from $v$ to $r_x$ in $G[Z_x]$ with $N_G(Y_x)\cap V(P)=\{v\}$,
\item for distinct integers $x,y\in \{1, \ldots, i\}$ with $x<y$, there are no edges between $Z_x$ and $Z_y\cup Y_y$. 
\end{itemize}
Assume that we have such $X_i$, $Y_1, \dots, Y_i, Z_1, \dots, Z_i$.  If $\chi(G[X_i])$ is sufficiently large, then some connected component $C$ of $G[X_i]$ has the same chromatic number as $G[X_i]$.
We choose a vertex $v$ in $C$, and 
we take a leveling $L_0, L_1, \ldots, L_m$ of $C$ where $L_i$ is the set of all vertices at distance $i$ from $v$.
Then there is a level $L_t$ such that $\chi(G[L_t])\ge \chi (G[X_i])/2$.
If $t=1$, we find a long induced cycle in $G[L_1]$ and thus we can obtain a large wheel vertex-minor directly.
Otherwise, it holds that $t\ge 2$.  Assign
 $X_{i+1}:=L_t$ and $Y_{i+1}:=L_{t-1}$ and $Z_{i+1}:=L_0\cup L_1\cup \cdots \cup L_{t-2}$.
Thus, by requiring $\chi(G[X_i])$ to be sufficiently large, we can either find $X_{i+1}$, $Y_1, \dots, Y_{i+1}, Z_1, \dots, Z_{i+1}$ with the desired properties, or find a large wheel vertex-minor.

From this structure, we will reduce to several types of simpler graphs step by step in Sections~\ref{sec:manufacturingwheel} and \ref{sec:patchedpath}.
We first find a long induced cycle $C=q_1q_2 \cdots q_mq_1$ in $G[X_i]$ using the result of Chudnovsky, Scott, and Seymour (Theorem~\ref{thm:longcycle}).
Secondly, we obtain a structure called a \emph{$(w, \ell)$-patched cycle} where $w=i$.
The definition will be rigorously given in Section \ref{sec:patchedpath}; for the moment, we proceed more informally.  A $(w,\ell)$-patched cycle consists of $C$ along with vertex sets $S_j=\{s^j_1, s^j_2, \ldots, s^j_\ell\}\subseteq Y_j$ for each $j\in \{1, 2, \ldots, i\}$ and 
a sequence of vertices $q_{b_1}, q_{b_2}, \ldots, q_{b_{\ell}}$ with $1\le b_1< b_2<\cdots< b_{\ell}\le m$ such that 
\begin{itemize}
\item for each $x\in \{1, \ldots, i\}$ and $y\in \{1,2,\ldots, \ell\}$, 
$s^x_y$ is adjacent to $q_{b_y}$
and non-adjacent to $q_{b_z}$ for all $z\in \{1, \ldots, \ell\}\setminus \{1, \ldots, y\}$.
\end{itemize}
We prove in Proposition~\ref{prop:intermediate} that the existence of this structure is guaranteed by assuming that the graph has no large wheel vertex-minor and the conditions that 
\begin{itemize}
\item $C$ is sufficiently long, 
\item for every $v\in V(C)$, $v$ has a neighbor in each $Y_j$, 
\item each vertex in $Y_1\cup Y_2\cup \cdots \cup Y_i$ has at most $n-1$ neighbors in $C$.
\end{itemize}
Concerning the final condition, if a vertex in $Y_1\cup Y_2\cup \cdots \cup Y_i$ has at least $n$ neighbors in $C$, then we can directly obtain a $W_n$ vertex-minor (Lemma~\ref{lem:fromlargerwheel}).

Up until this point in the argument, we have made no assumptions on the possible edges between pairs of vertices in the set $S_1\cup S_2\cup \cdots \cup S_i$.
As we argued in Figure~\ref{fig:badcasenewway}, we want to find a large independent set formed by two disjoint subsets from two distinct sets $S_j$ and $S_{j'}$. For this, we apply a Ramsey-type argument, which we call the \emph{rectangular Ramsey lemma} (Proposition~\ref{prop:recramsey}).
This lemma implies that there exist a large subset $J\subseteq \{1, 2, \ldots, \ell\}$ and $\{c_1, c_2\}\in \{1, 2, \ldots, i\}$ such that
$\{s^{c_1}_x: x\in J\}\cup \{s^{c_2}_x: x\in J\}$ is an independent set, if $G$ has no large clique. 

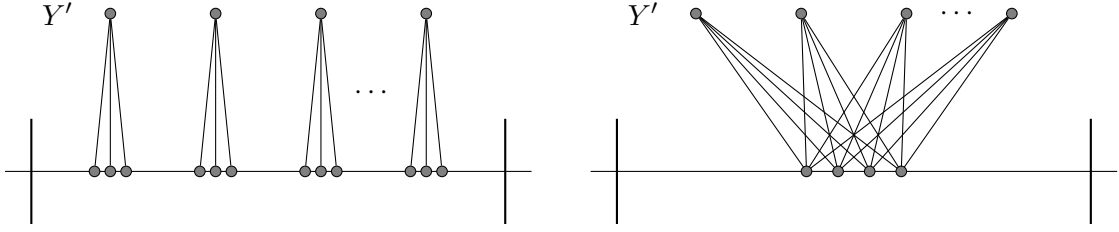
\begin{figure}
  \centering
  \begin{tikzpicture}[scale=0.7]
    \tikzstyle{w}=[circle,draw,fill=black!50,inner sep=0pt,minimum width=4pt]

   \draw (0,0)--(10,0);
  \draw[thick] (0.5, 1)--(0.5,-1);
  \draw[thick] (9.5, 1)--(9.5,-1);
 \foreach \y in {1,...,4}{
      \draw (\y*2-.3,0) node [w] (a\y) {};
      \draw (\y*2+.3,0) node [w] (b\y) {};
      \draw (\y*2,0) node [w] (c\y) {};
      \draw (\y*2,3) node [w] (e\y) {};
      \draw(a\y)--(e\y);
      \draw(b\y)--(e\y);
      \draw(c\y)--(e\y);
    }

      \node at (1, 3) {$Y'$};
      \node at (7, 1.5) {$\cdots$};

   \end{tikzpicture}\quad\quad
  \begin{tikzpicture}[scale=0.7]
    \tikzstyle{w}=[circle,draw,fill=black!50,inner sep=0pt,minimum width=4pt]

   \draw (0,0)--(10,0);
  \draw[thick] (0.5, 1)--(0.5,-1);
  \draw[thick] (9.5, 1)--(9.5,-1);
 \foreach \y in {1,...,4}{
      \draw (\y*2,3) node [w] (e\y) {};
      \draw(5-.9, 0)--(e\y);
      \draw(5-.3, 0)--(e\y);
      \draw(5+.3, 0)--(e\y);
      \draw(5+.9, 0)--(e\y);
    }
      \draw (5-.9,0) node [w] () {};
      \draw (5-.3,0) node [w] () {};
      \draw (5+.3,0) node [w] () {};
      \draw (5+.9,0) node [w] () {};

      \node at (1, 3) {$Y'$};
      \node at (7, 3) {$\cdots$};

   \end{tikzpicture}
  \caption{The intended application of the regular partition lemma.}\label{fig:partitionlemma}
\end{figure}

We further refine the adjacency relations between $\{s^{c_1}_x: x\in J\}\cup \{s^{c_2}_x: x\in J\}$ and $C$ using the following Ramsey-type argument:
for a graph $H$ on the vertex set $D\cup Y$ such that
\begin{itemize}
\item $H[D]$ is a sufficiently long induced cycle, 
\item for every $v\in D$, $v$ has a neighbor in $Y$, 
\item each vertex in $Y$ has at most $n-1$ neighbors in $D$,
\end{itemize}
there is a large subset $Y'\subseteq Y$ and a partition of $H[D]$ into at most $n-1$ paths such that for each part, either vertices in $Y'$ 
have the exactly same neighborhood, or neighborhoods appear in a consecutive order.
Figure~\ref{fig:partitionlemma} shows the two cases for how the vertices of $Y'$ can be adjacent to the vertices in $H[D]$ which is in one subpath of an element of the partition of $H[D]$.  
We prove this result in a more general setting, which we call the \emph{regular partition lemma} (Proposition~\ref{prop:regularpartition}), with the hope that it might be of use in other situations. 

Depending on the outcome of the application of the regular partition lemma, we show that $G$ contains a vertex-minor isomorphic to one of several cases we call a drum, a clam, and a hanging ladder, depicted in Figures~\ref{fig:drum}, \ref{fig:clam}, \ref{fig:hanging}, respectively.

\section{Regular partition lemma}\label{sec:regpartlem}

For a sequence $(A_1, \ldots, A_{\ell})$ of finite subsets of an interval $I\subseteq \mathbb{R}$, a partition $\{I_1, \ldots, I_k\}$ of $I$ into intervals 
is called a \emph{regular partition} of $I$ with respect to $(A_1, \ldots, A_{\ell})$ if 
for all $i\in\{1,\ldots,k\}$, either 
\begin{itemize}
\item $A_1\cap I_i=A_2\cap I_i=\cdots = A_{\ell}\cap I_i\not= \emptyset$, or 
\item $\abs{A_1\cap I_i}=\abs{A_2\cap I_i}=\cdots = \abs{A_{\ell}\cap I_i}>0$, and for all $j,j'\in \{1, \ldots, \ell\}$ with $j<j'$, $\max(A_j\cap I_i)<\min (A_{j'}\cap I_i)$, or
\item $\abs{A_1\cap I_i}=\abs{A_2\cap I_i}=\cdots = \abs{A_{\ell}\cap I_i}>0$, and for all $j, j'\in \{1, \ldots, \ell\}$ with $j<j'$, $\max (A_{j'}\cap I_i)< \min (A_j\cap I_i)$. 
\end{itemize}
The number of parts $k$ is called the \emph{order} of the regular partition.

The following lemma is a strengthening of Erd\H os-Szekeres theorem. We simply follow the proof of Seidenberg~\cite{seidenberg1959}. We say that a sequence is \emph{identical} if all elements of the sequence are same. 
\begin{lemma}\label{lem:erdosszekerestype}
For every sequence $(a_1, \ldots, a_{(\ell-1)^3+1})$ of real numbers, there exists a subsequence $(a_{i_1}, \ldots, a_{i_\ell})$ that is identical or strictly increasing or strictly decreasing. %
\end{lemma}

\begin{proof}
For each $a_i$, we define a triplet $(a_i^1, a_i^2, a_i^3)$ where
\begin{itemize}
\item $a_i^1$ is the length of the longest identical subsequence ending at $a_i$, %
\item $a_i^2$ is the length of the longest strictly increasing subsequence ending at $a_i$, and
\item $a_i^3$ is the length of the longest strictly decreasing subsequence ending at $a_i$.
\end{itemize} 

Note that $(a_i^1, a_i^2, a_i^3)\neq(a_j^1, a_j^2, a_j^3)$ for all $i\neq j$, since $a_j=a_i$ or $a_j>a_i$ or $a_j<a_i$. 
However, the number of different triplets such that $0< a_i^1, a_i^2, a_i^3<\ell$ is at most $(\ell-1)^3$. 
Therefore, there exists $a_k$ such that one of $a_k^1, a_k^2$, and $a_k^3$ is $\ell$, completing the proof. 
\end{proof}

\begin{proposition}[Regular partition lemma]\label{prop:regularpartition}
Let $I\subseteq \mathbb R$ be an interval.
For all positive integers $k$ and $\ell$,
there exists a positive integer $N = N(k, \ell)$ satisfying the following.
For every sequence $(A_1, \ldots, A_N)$ of $k$-element subsets of $I$, 
there exist a subsequence $(A_{j_1}, \ldots, A_{j_{\ell}})$ of $(A_1, \ldots, A_N)$ and a regular partition of $I$ with respect to $(A_{j_1}, \ldots, A_{j_{\ell}})$ that has order at most $k$. 
\end{proposition} 

\begin{proof} 

We recursively define $t(n,\ell), M(n,\ell), N(n, \ell)$ as follows
\begin{align*}
t(n,\ell)&=
  \begin{cases}
    \max \{N(i,N(n-i,\ell)): i=1,2,\ldots,n-1\}  &\text{if }n>1,
    \\
    \ell &\text{if }n=1.
  \end{cases}\\
M(n,\ell)&=(t(n,\ell)-1)(\ell-1)n+1.\\
N(n,\ell)&=
  \begin{cases}
    (M(n,\ell)-1)^{(2^n)}+1& \text{if }n>1,\\
    (\ell-1)^3+1 & \text{if } n=1.
  \end{cases}
\end{align*}
We proceed by induction on $k$. 
If $k=1$, the statement is implied by Lemma~\ref{lem:erdosszekerestype}. 
If $\ell=1$, then the partition $\{I\}$ of $I$ is a regular partition with respect to $A_1$. 
We may assume that $k,\ell\ge 2$. Note that $M(k,\ell)\ge t(k,\ell)$. 
By slightly abusing notation, let us identify $A_i$ with a strictly increasing sequence $(a_{i,1}, a_{i,2},\ldots, a_{i,k})$ 
of its elements.
Let $M=M(k,\ell)$, and $t=t(k,\ell)$. 
Let $\mathcal A_0=(A_1,\ldots,A_N)$. 
For each $i=1,2,\ldots,k$, there exists a subsequence $\mathcal A_i$ of $\mathcal A_{i-1}$ such that
the sequence of 
the $i$-th elements of terms of $\mathcal A_i$  is (not necessarily strictly) increasing or decreasing 
where \[
\abs{\mathcal A_{i}} \ge \sqrt{\abs{\mathcal A_{i-1}}-1}+1
\]
by the Erd\H{o}s-Szekeres Theorem. Then $\abs{\mathcal A_k}\ge M$. Let $\mathcal A=(A_{i_1}, \ldots, A_{i_M})$ be a subsequence of $\mathcal A_k$ of length $M$.

By the construction, for each $j\in\{1,2,\ldots,k\}$, 
the sequence $(a_{i_1,j},a_{i_2,j},\ldots,a_{i_M,j})$ of
the $j$-th elements 
of terms of $\mathcal A$ is increasing
or decreasing.
By symmetry, we may assume that
$(a_{i_1,1}, \ldots, a_{i_M,1})$ is increasing, 
because otherwise we consider the reverse $(A_N, A_{N-1},\ldots,A_1)$.

Suppose that there exists an integer $0<j<k$ such that  $(a_{i_1,j}, \ldots, a_{i_M,j})$ is increasing and  $(a_{i_1,j+1}, \ldots, a_{i_M,j+1})$ is decreasing. Let  $x\in (a_{i_M,j}, a_{i_M,j+1})$.
As $(a_{i_1,j}, \ldots, a_{i_M,j})$ is increasing,
the first $j$ elements of each term of $\mathcal A$ are less than $x$.
Similarly the remaining $k-j$ elements of each term of $\mathcal A$
are greater than $x$
because $(a_{i_1,j+1}, \ldots, a_{i_M,j+1})$ is decreasing.
Thus we observe that
\[\abs{A_{i_1}\cap (-\infty,x]}=\abs{A_{i_2}\cap (-\infty,x]}
=\cdots=\abs{A_{i_M}\cap (-\infty,x]}=j\]
and 
\[\abs{A_{i_1}\cap (x,\infty)}=\abs{A_{i_2}\cap (x,\infty)}
=\cdots=\abs{A_{i_M}\cap (x,\infty)}=k-j.\]
Since $0<j<k$ and $M\ge t\ge N(j, N(k-j,\ell))$, by the induction hypothesis
applied to $(A_{i_1}\cap (-\infty,x], A_{i_2}\cap (-\infty,x],\ldots,
A_{i_M}\cap (-\infty,x])$,
there exist a subsequence $\mathcal A'$ of $\mathcal A$ with $\abs{\mathcal A'}=N(k-j,\ell)$ and a regular partition of $I\cap (-\infty,x]$ with respect to $\mathcal A'$ that has order at most $j$. Again by the induction hypothesis, we obtain a subsequence $\mathcal A''$ of $\mathcal A'$ with $\abs{\mathcal A''}=\ell$ and a regular partition of $I\cap (x,\infty)$ with respect to $\mathcal A''$ that has order at most $k-j$. The union of the regular partitions of $I\cap (-\infty,x]$ and $I\cap (x,\infty)$ is a regular partition of $I$ with respect to $\mathcal A''$ of order at most $k$, so we are done.
Therefore, we may assume that $(a_{i_1,j}, \ldots, a_{i_M,j})$ is increasing for every $j\in \{1,\ldots, k\}$.

Suppose that
$a_{i_{s+t-1},j}< a_{i_s,j+1}$ for  some $1\le s\le M-t+1$ and $1\le j\le k-1$.
Then
there exists  $x\in (a_{i_{s+t-1},j}, a_{i_s,j+1})$.
We deduce that
\[ \abs{A_{i_s}\cap (-\infty,x]}=
  \abs{A_{i_{s+1}}\cap(-\infty,x]}=
    \cdots
    =\abs{A_{i_{s+t-1}}\cap (-\infty,x]}=j\]
  and 
\[ \abs{A_{i_s}\cap (x,\infty)}=
  \abs{A_{i_{s+1}}\cap(x,\infty)}=
    \cdots
    =\abs{A_{i_{s+t-1}}\cap (x,\infty)}=k-j.\]
  Since $t\ge N(j,N(k-j,\ell))$, 
  by applying the induction hypothesis to a partition $I\cap (-\infty,x]$
and to a partition $I\cap (x,\infty)$ as in the previous paragraph, 
we are done.

Thus, we may assume that
$a_{i_{s+t-1},j}\ge a_{i_s,j+1}$ for  all $1\le s\le M-t+1$ and $1\le j\le k-1$.
Therefore $a_{i_{s+(t-1)k},1}\ge a_{i_{s+t-1},k}>a_{i_{s+t-1},k-1}\ge a_{i_s,k}$ for all $s$ with $1\le s\le M-(t-1)k$. This implies that
$\max A_{i_{1+(t-1)kj}}<\min A_{i_{1+(t-1)k(j+1)}}$ for each $0\le j\le \ell-2$
  and therefore 
$\{I\}$ is a regular partition of $I$ with respect to $(A_{i_1}, A_{i_{1+(t-1)k}}, \ldots, A_{i_{1+(t-1)k(\ell-1)}})$. 
\end{proof}

\begin{corollary}\label{cor:regularpartition}
Let $I$ be an interval in $\mathbb R$.
For all positive integers $k$ and $\ell$, there exists an integer $N=N'(k, \ell)$ satisfying the following:
For every sequence $(A_1, \ldots, A_N)$ of sets of at most $k$
reals in $I$, 
there exist a sequence $1\le j_1<j_2<\cdots<j_\ell\le N$ and a regular partition of $I$ with respect to $(A_{j_1}, \ldots, A_{j_{\ell}})$ that has order at most $k$. 
\end{corollary}

\section{Rectangular Ramsey Lemma}\label{sec:ramsey}
Let $G$ be a graph with vertex set $\{1,2,\ldots,m\}\times\{1,2,\ldots,n\}$.
We would like to show that either $G$ has a large clique
or there exist subsets $X\subseteq\{1,2,\ldots,m\}$ and $Y\subseteq\{1,2,\ldots,n\}$ 
such that $X\times Y$ is an independent set in $G$ and both $\abs{X}$ and $\abs{Y}$ are large.
We prove it using the Product Ramsey Theorem.
For a set $X$ and a non-negative integer $k$, we denote by ${X\choose k}$ the set of all $k$-element subsets of $X$.

\begin{theorem}[Theorem 11.5 of \cite{Trotter1992}; See also \cite{GrahamRS1990}]\label{thm:productramsey}
Let $r, t$ be positive integers, and let $k_1, k_2, \ldots, k_t$ be nonnegative integers, and
let $m_1, m_2, \ldots, m_t$ be integers with $m_i\ge k_i$ for each $i\in \{1, 2, \ldots, t\}$.
Then there exists an integer $R=R_{prod}(r,t;k_1, k_2, \ldots, k_t; m_1, m_2, \ldots, m_t)$ such that
if $X_1, X_2, \ldots, X_t$ are sets with $\abs{X_i}\ge R$ for each $i\in \{1, 2, \ldots, t\}$, 
then for every function $f:{X_1 \choose k_1}\times {X_2 \choose k_2} \times \cdots \times {X_t \choose k_t} \rightarrow \{1, 2, \ldots, r\}$, 
there exist an element $\alpha\in \{1, 2, \ldots r\}$ and subsets $Y_1, Y_2, \ldots, Y_t$ of $X_1, X_2, \ldots, X_t$, respectively, 
so that $\abs{Y_i}\ge m_i$ for each $i\in \{1, 2, \ldots, t\}$, and 
$f$ maps every element of ${Y_1\choose k_1}\times {Y_2\choose k_2}\times \cdots \times {Y_t\choose k_t}$ to $\alpha$.
\end{theorem}

\begin{proposition}\label{prop:recramsey}
  For all positive integers $a$, $b$, and $k$, there exist positive integers $M=R_1(a,b,k)$ and $N=R_2(a,b,k)$ such that 
  for all $m\ge M$ and $n\ge N$, 
  every graph $G$ on $\{1,2,\ldots,m\}\times\{1,2,\ldots,n\}$ 
  either has  a clique of $k$ vertices
  or has subsets $X\subseteq\{1,2,\ldots,m\}$ and $Y\subseteq \{1,2,\ldots,n\}$ such that 
  $\abs{X}=a$, $\abs{Y}=b$, and $X\times Y$ is an independent set in $G$.
\end{proposition}
\begin{proof}
 Let $t:=\max(a, b, k)$, and let $M=N=R_{prod}(7, 2; 2, 2; t, t)$. 
  We may assume that $G$ is a graph on $\{1,2,\ldots,M\}\times
  \{1,2,\ldots,N\}$.
  Let us write $v_{ij}=(i,j)$ to denote a vertex.

We define a function $f: {\{1, 2, \ldots, M\}\choose 2}\times {\{1, 2, \ldots, N\}\choose 2}\rightarrow \{1, 2, \ldots, 7\}$ as follows.
For $\{x_1, x_2\}\subseteq \{1, 2, \ldots, M\}$ and $\{y_1, y_2\}\subseteq \{1, 2, \ldots, N\}$ with $x_1<x_2$ and $y_1<y_2$, 
let 
\[(e_1, e_2, e_3, e_4, e_5, e_6):=(v_{x_1y_1}v_{x_2y_1}, v_{x_1y_1}v_{x_1y_2}, v_{x_1y_1}v_{x_2y_2},  v_{x_2y_1}v_{x_1y_2}, v_{x_2y_1}v_{x_2y_2}, v_{x_1y_2}v_{x_2y_2}),\]
and let $f(\{x_1, x_2\}, \{y_1, y_2\}):=i$ if $G$ contains $e_i$ but does not contain $e_j$ for all $j<i$, and $f(\{x_1, x_2\}, \{y_1, y_2\}):=7$ if $G$ contains no edges in $\{e_1, \ldots, e_6\}$.
By Theorem~\ref{thm:productramsey}, 
there exist $\alpha \in \{1, 2, \ldots, 7\}$, $X\subseteq \{1, 2, \ldots, M\}$, and $Y \subseteq \{1, 2, \ldots, N\}$ with $\abs{X}\ge t$ and $\abs{Y}\ge t$ such that
$f$ maps every element of ${X\choose 2}\times {Y\choose 2}$ to $\alpha$.
If $\alpha=7$, then $X\times Y$ is an independent set with $\abs{X}\ge a$ and $\abs{Y}\ge b$, 
and if $\alpha\in \{1, 2, \ldots, 6\}$, then $G$ contains a clique of size $t\ge k$.
\end{proof}

\section{Manufacturing wheels}\label{sec:manufacturingwheel}
We will use the following Ramsey-type result on connected graphs.

\begin{theorem}[folklore; see Diestel~{\cite{Diestel2010}}]\label{thm:binary}
  For $k\ge 1$ and $\ell\ge 3$,  
  every connected graph on at least $k^{\ell-2}+1$ vertices contains 
  a vertex of degree at least $k$ or an induced path on $\ell$ vertices.
\end{theorem}

The following lemma is useful to find an induced matching in a bipartite graph.

\begin{lemma}[Lemma 7.8 of \cite{KwonO2014}]\label{lem:bipartitielemma}
Let $n$ be a positive integer.
Let $G$ be a bipartite graph with a bipartition $(A,B)$ such that 
\begin{itemize}
\item every vertex in $A$ has a neighbor, 
\item every vertex in $B$ has at most $n$ neighbors.
\end{itemize}
Then there is an induced matching of size at least $\abs{A}/n$.
\end{lemma}

For every integer $n\ge 3$, let $\mu(n)=(n-1)(R(n,n)^{2n-3}+1)$. 
Lemma~\ref{lem:reduceconnected} is useful to reduce the size of a connected subgraph.

\begin{lemma}[Choi, Kwon, and Oum~\cite{ChoiKO2016}]\label{lem:connected}
Let $H$ be a connected graph with at least $2$ vertices. 
For each vertex $v$ of $H$, either $H- v$ or $H*v- v$ is connected. 
\end{lemma}

\begin{lemma}\label{lem:reduceconnected}
Let $n\ge 3$ be an integer and let $G$ be a graph on the vertex set $A\cup U\cup S$ such that
\begin{enumerate}
\item $A$, $U$, and $S$ are pairwise disjoint,  
\item there are no edges between $A$ and $S$,
\item $U$ is an independent set,
\item each vertex in $U$ has a neighbor on $S$, and
\item there exists a vertex $w\in S$ where 
for every $v\in N_G(U)\cap S$, there is a path $P$ from $v$ to $w$ in $G[S]$ with $N_G(U)\cap V(P)=\{v\}$.
\end{enumerate}
If $\abs{U}\ge \mu(n)$, then there exist $U'\subseteq U$ with $\abs{U'}\ge n$ and $v\in S$ and a graph $G'$ on $A\cup U'\cup \{v\}$ such that
\begin{enumerate}
\item $G'[A\cup U']=G[A\cup U']$,
\item $v$ is adjacent to all vertices in $U'$ and has no neighbors in $A$ in $G'$, 
\item $G'$ is a vertex-minor of $G$.
\end{enumerate}
\end{lemma}
\begin{proof}
Let $m:=\mu(n)$ and let $U:=\{u_1, u_2, \ldots, u_m\}$.
For each $v\in N_G(U)\cap S$, let $P_v$ be a path from $v$ to $w$ in $G[S]$ with $N_G(U)\cap V(P_v)=\{v\}$.

Suppose $w\in N_G(U)\cap S$. By the assumption that for every $v\in N_G(U)\cap S$ it holds that $N_G(U)\cap V(P_v)=\{v\}$, 
we conclude that $N_G(U)\cap S=\{w\}$. Therefore, $w$ is adjacent to all vertices in $U$, and we are done.
We may assume that $w\in S\setminus N_G(U)$.

If there is a vertex in $S$ having at least $n$ neighbors on $U$, 
then we are done.
We may assume that every vertex in $S$ has at most $n-1$ neighbors on $U$.
By Lemma~\ref{lem:bipartitielemma}, there exist a subset $\{a_1, a_2, \ldots, a_{m/(n-1)}\}$ of $\{1, 2, \ldots, m\}$ and a subset $\{s_{1}, s_{2}, \ldots, s_{m/(n-1)}\}$ of $S$ such that 
\begin{itemize}
\item $u_{a_i}$ is adjacent to $s_j$ if and only if $i= j$.
\end{itemize}
Let $U_1:=A\cup \{u_{a_i}: 1\le i\le m/(n-1)\}$ and
$U_2:=\{s_i: 1\le i\le m/(n-1)\}$. 
Let $G_1:=G[U_1\cup U_2\cup (\bigcup_{v\in U_2} V(P_v))]$.
Note that $N_{G_1}(U_1)=U_2$, $G_1-U_1$ is connected, and every vertex in $V(G_1)\setminus (U_1\cup U_2)$ has no neighbors in $U_1$.

Choose a sequence of graphs $H_1, H_2, \ldots, H_y$ such that
\begin{enumerate}[(1)]
\item $H_1=G_1$, 
\item $V(H_y)=U_1\cup U_2$, 
\item for each $i\in \{1, 2, \ldots, y-1\}$, $H_{i+1}=H_i-v_i$ or $H_{i+1}=H_i*v_i-v_i$ for some $v_i\in V(H_i)\setminus (U_1\cup U_2)$, 
\item for each $i\in \{1, 2, \ldots, y\}$, $H_{i}-U_1$ is connected.
\end{enumerate}
We claim that such a sequence always exists. Let $H_1, H_2, \ldots, H_{y'}$ be a maximal sequence satisfying (1), (3), and (4).
Assume, to reach a contradiction, that $V(H_{y'})\neq U_1\cup U_2$. By the assumptions, we have $U_1\cup U_2\subseteq V(H_{y'})$, and therefore, $V(H_{y'})\setminus (U_1\cup U_2)\neq \emptyset$.
Let $v_{y'}\in V(H_{y'})\setminus (U_1\cup U_2)$.
Since $H_{y'}-U_1$ is connected, 
by Lemma~\ref{lem:connected}, 
$(H_{y'}-U_1)-v_{y'}$ or $(H_{y'}-U_1)*v_{y'}-v_{y'}$ is connected.
We fix $H_{y'+1}$ to be one of $(H_{y'}-U_1)-v_{y'}$ and $(H_{y'}-U_1)*v_{y'}-v_{y'}$ which is connected.  This contradicts the maximality of the sequence.
We conclude that there exists a sequence $H_1, H_2, \ldots, H_y$ satisfying (1) - (4). Let $G_2:=H_y$.
Note that 
\begin{itemize}
\item $V(G_2)=U_1\cup U_2$, 
\item $G_2[U_2]$ is connected, 
\item $G_2-E(G_2[U_2])=G_1[U_1\cup U_2]-E(G_1[U_2])=G[U_1\cup U_2]-E(G[U_2])$.
\end{itemize}

Since $\abs{U_2}=m/(n-1)= R(n,n)^{2n-3}+1$, by Theorem~\ref{thm:binary},  
$G_2[U_2]$ contains either a vertex of degree at least $R(n,n)$ or an induced path on $2n-1$ vertices. 

Suppose $G_2[U_2]$ contains a vertex $s_j$ of degree at least $R(n,n)$ for some $1\le j\le m/(n-1)$. 
Since $s_j$ has at least $R(n,n)$ neighbors in $U_2$, $N_{G_2}(s_j)\cap U_2$ 
contains either a clique of size $n$ or an independent set of size $n$. 
We define \[
    G_3:=
    \begin{cases}
      G_2-u_{a_j}&\text{if $N_{G_2}(s_j)\cap U_2$ contains an independent set of size $n$},\\
      (G_2-u_{a_j})*s_j &\text{otherwise}.
    \end{cases}
  \]
Note that $N_{G_3}(s_j)\cap U_2$ contains an independent set of size $n$. 
Let $\{s_{d_1}, s_{d_2}, \ldots, s_{d_n}\}$ be an independent set in $N_{G_3}(s_j)\cap U_2$.
Note that the application of a local complementation when we obtain $G_3$ does not change the adjacency relation between 
$\{s_{d_1}, s_{d_2}, \ldots, s_{d_n}\}$ and $\{u_{a_{d_1}}, u_{a_{d_2}}, \ldots, u_{a_{d_n}}\}$ as $s_j$ has no neighbors on $\{u_{a_{d_1}}, u_{a_{d_2}}, \ldots, u_{a_{d_n}}\}$ in $G_2$.
Let 
 \[G':=(G_3*s_{d_1}*s_{d_2}*\cdots *s_{d_n})[A\cup \{u_{a_{d_i}}:1\le i\le n\}\cup \{s_j\}].\]
 Then 
 we have $G'[A\cup \{u_{a_{d_i}}:1\le i\le n\}]=G[A\cup \{u_{a_{d_i}}:1\le i\le n\}]$, and
$s_j$ is adjacent to all vertices in $\{u_{a_{d_i}}:1\le i\le n\}$ and has no neighbors in $A$ in $G'$, 
as required.

Suppose $G_2[U_2]$ contains an induced path $s_{i_1}s_{i_2} \ldots s_{i_{2n-1}}$. 
Let 
\[G':=(G_3*s_{i_1}*s_{i_2}* \cdots *s_{i_{2n-2}})-\{s_{i_1}, s_{i_2}, \ldots, s_{i_{2n-2}}\}- \{u_{i_2}, u_{i_4}, u_{i_6}, \ldots, u_{i_{2n-2}}\}.\] 
Then in $G'$, $s_{i_{2n-1}}$ is adjacent to all of $u_{i_1}, u_{i_3}, u_{i_5}, \ldots, u_{i_{2n-1}}$, and 
$\{u_{i_1}, u_{i_3}, u_{i_5}, \ldots, u_{i_{2n-1}}\}$ is an independent set of size $n$, and 
$s_{i_{2n-1}}$ has no neighbor in $A$, as required.
\end{proof}

\subsection{From a partial wheel with many spokes}
\begin{lemma}\label{lem:fromlargerwheel}
  Let $n\ge 3$ be an integer and let $G$ be a graph such that $G-v$ is an induced cycle of length
  at least $n+3$. 
  If the degree of $v$ is at least $n$, then $G$ has $W_n$ as a vertex-minor.
\end{lemma}

\begin{proof}
Let $s$ be the length of the induced cycle $G-v$. Let $v_1, \ldots, v_s$ be the vertices of the cycle in a cyclic order and let $v_{s+\ell}:=v_{\ell}$ for $1 \le \ell \le s$. Let $t$ be the degree of $v$. Note that $s\ge t$. 

We prove by induction on $s+t$. 
The following statements cover base cases which are either $t=n$ or $s=n+3$:
\begin{itemize} 
\item (Case 1. $t=n$.) Let $v_{i_1}, v_{i_2}, \ldots, v_{i_{s-t}}$ be the vertices of $G-v$ that are non-adjacent to $v$ in $G$. The graph obtained from $G$ by smoothing $v_{i_1}, v_{i_2}, \ldots, v_{i_{s_t}}$ is isomorphic to $W_n$. 
\item (Case 2. $s=n+3$ and $t=n+1$.) Let $v_i, v_j$ be the vertices in $G-v$ that are non-adjacent to $v$. Note that we may assume that $i\neq j+1$ and $i\neq j+2$ by symmetry. The graph $(G*v_i*v_{j+1}*v_{j+2})-\{v_i, v_{j+1}, v_{j+2}\}$ is isomorphic to $W_n$.  
\item (Case 3. $s=n+3$ and $t=n+2$.) Let $v_i$ be the vertex in $G-v$ that is non-adjacent to $v$. The graph $(G*v_{i+1}*v_i*v_{i-1})-\{v_{i-1}, v_i, v_{i+1}\}$ is isomorphic to $W_n$. 
\item (Case 4. $s=n+3$ and $t=n+3$.) Let $v_i$ be a vertex in $G-v$. The graph $(G*v_i*v_{i-1}*v_{i+1})-\{v_{i-1}, v_i, v_{i+1}\}$ is isomorphic to $W_n$. 
\end{itemize}
We may assume that $s>n+3$ and $t>n$. 
Suppose that $s>t$. There exists $i$ such that $v$ is not adjacent to $v_i$ and adjacent to $v_{i+1}$. Let $G_1=G*v_{i}-v_{i}$. Then $G_1-v$ is an induced cycle of length $s-1\ge n+3$ and the degree of $v$ is $t$ in $G_1$. By induction hypothesis, $G_1$ has $W_n$ as a vertex-minor, which implies that $G$ has $W_n$ as a vertex-minor. 
Now, we may assume that $s=t>n+3$. For a vertex $u$ in $G-v$, let $G_2=G*u-u$. Then $G_2-v$ is an induced cycle of length $s-1\ge n+3$ and the degree of $v$ is $t-3\ge n$ in $G_2$. By induction hypothesis, $G_2$ has $W_n$ as a vertex-minor, which implies that $G$ has $W_n$ as a vertex-minor. 
\end{proof}

\subsection{From drums, clams, and hanging ladders}
\begin{figure}
  \centering
  \begin{tikzpicture}[yscale=0.5]
    \tikzstyle{every node}=[circle,draw,fill=black!50,inner sep=0pt,minimum width=4pt]
    \draw (0,0) circle (2);
    \foreach \y in {0,1,2,3,4,5,6} {
      \draw (180+30*\y:2) node (c\y) {}
      --+ (0,-1) node {}
      --+ (0,-2) node (d\y) {};
    }
    \draw (d0)--(d1)--(d2)--(d3)--(d4)--(d5)--(d6);
  \end{tikzpicture}
  \caption{A drum}
  \label{fig:drum}
\end{figure}

A \emph{drum} on $3n$ vertices 
is the graph on the vertex set $\{v_1,v_2,\ldots,v_n,w_1,w_2,\ldots,w_n,u_1,u_2,\ldots,u_n\}$ such that $v_1v_2\cdots v_n$ is an induced path, 
$w_1w_2\cdots w_n$ is an induced cycle, 
each $u_i$ is adjacent only to $v_i$ and $w_i$,
and there are no edges between $\{v_1,v_2,\ldots,v_n\}$ and $\{w_1,w_2,\ldots,w_n\}$. See Figure~\ref{fig:drum} for an illustration.

\begin{lemma}\label{lem:drum}
  For an integer $n\ge 3$, 
  a drum on $3(2n-1)$ vertices has $W_n$ as a vertex-minor.
\end{lemma}
\begin{proof}
  Let $G$ be a drum on $3(2n-1)$ vertices with the vertex labels as in the definition of drums.
  Let $H=G*v_1*v_2*v_3\cdots*v_{2n-2}$.
  Then, in $H$, $v_{2n-1}$ is adjacent to all of $u_1$, $u_2$, $\ldots$, $u_{2n-1}$. Furthermore $\{u_1,u_3,u_5,\ldots,u_{2n-1}\}$ is an independent set in $H$. Thus, 
  \[H-\{u_2, u_4, u_6, \ldots, u_{2n-2}, v_1, v_2, \ldots, v_{2n-2}\}\] is isomorphic to a subdivision of $W_n$.
\end{proof}

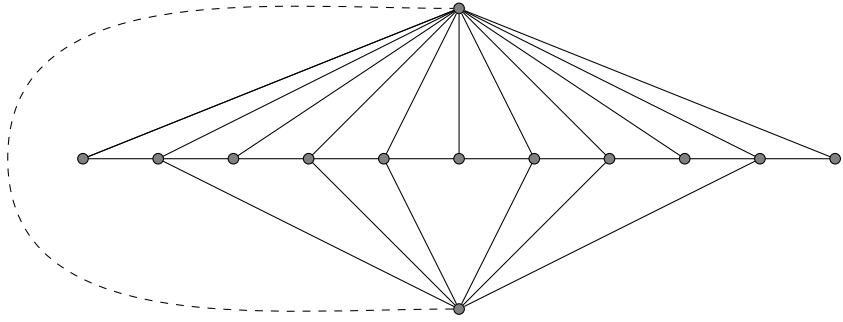
\begin{figure}
  \centering
  \begin{tikzpicture}
    \tikzstyle{every node}=[circle,draw,fill=black!50,inner sep=0pt,minimum width=4pt]
    \draw (5,2) node (c){}--(0,0)--(10,0);
    \draw (5,-2) node (z){};
    \foreach \z in {1,3,4,6,7,9}{
      \draw (\z,0)--(z);
      }
    \foreach \y in {0,...,10}{
      \draw (\y,0) node (c\y) {} -- (c);
    }
    \draw[dashed] (z) [out=180]to [in=-90] (-1,0)[out=90] to [in=180](c);
  \end{tikzpicture}
  \caption{A clam. (A dashed line may be an edge or not.)}\label{fig:clam}
\end{figure}
For an integer $n\equiv 1\pmod 3$, 
a \emph{clam} on $n$ vertices 
is a graph on the vertex set $\{v_1,v_2,\ldots,v_{n-2},h_1,h_2\}$
such that $v_1v_2\ldots v_{n-2}$ is an induced path, 
$h_1$ is adjacent to all of $v_1$, $v_2$, $\ldots$, $v_{n-2}$, 
and $h_2$ is adjacent to $v_i$ if and only if $1<i<n-2$ and $i\not\equiv 0\pmod 3$.  Note that $h_1$ and $h_2$ may be adjacent in a clam.

\begin{lemma}\label{lem:clam}
  For an integer $n\ge 2$, a clam on $3n+4$ vertices contains $W_{2n}$ or $W_{2n+1}$ as a vertex-minor.
\end{lemma}

\begin{proof}
  Let $G$ be a clam on $3n+4$ vertices 
  with the vertex labels as in the definition.
  Let $H=G*v_3*v_6*v_9*\cdots*v_{3n}$.
  In $H$, $h_1$ is non-adjacent to $\{v_2, v_4, v_5, v_7, \ldots, v_{3n-1}, v_{3n+1}\}$,
  and $h_1v_1v_2v_4v_5v_7v_8\cdots v_{3n-1}v_{3n+1}v_{3n+2}h_1$ is an induced cycle of length $2n+3$.
  Still in $H$, $h_2$ is adjacent to $2n$ vertices among $v_1$, $v_2$, $v_4$, $v_5$, $v_7$, $v_8$, $\ldots$, $v_{3n-1}$, $v_{3n+1}$, $v_{3n+2}$.
  Thus, $H-\{v_3, v_6, v_9, \cdots, v_{3n}\}$ is isomorphic to a subdivision of $W_{2n}$ or $W_{2n+1}$, depending on whether or not $h_2$ is adjacent to $h_1$ in $H$.
\end{proof}

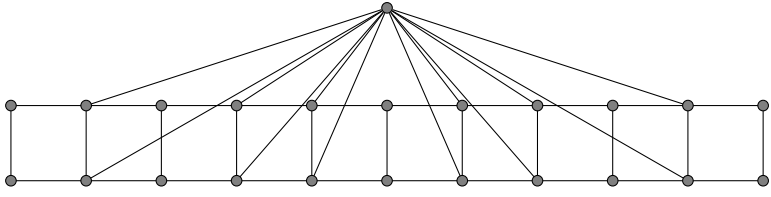
\begin{figure}
  \centering
  \begin{tikzpicture}
    \tikzstyle{every node}=[circle,draw,fill=black!50,inner sep=0pt,minimum width=4pt]
    \draw (5,2.3) node (c){};
    \draw (0,0)--(10,0);
    \draw (0,1)--(10,1);
    \foreach \z in {1,3,4,6,7,9}{
      \draw (\z,0)--(c)--(\z,1);
      }
    \foreach \y in {0,...,10}{
      \draw (\y,0) node (c\y) {}--
      (\y,1) node (d\y) {};
    }
  \end{tikzpicture}
  \caption{A hanging ladder}
  \label{fig:hanging}
\end{figure}
A \emph{hanging ladder} on $6n+5$ vertices is a graph on the vertex set $\{v_1,v_2,\ldots,v_{3n+2},w_1,w_2,\ldots,w_{3n+2},c\}$ such that
\begin{itemize}
\item 
$v_i$ is adjacent to $w_j$ if and only if $i=j$, 
\item $v_1v_2\ldots v_{3n+2}$ and $w_1w_2\ldots w_{3n+2}$ are induced paths,
\item the set of neighbors of $c$ is 
  $\{v_i,w_i: 1<i<3n+2, i\not\equiv 0\pmod 3\} $.
\end{itemize}
\begin{lemma}\label{lem:hangingladder}
  For an integer $n\ge 2$, a hanging ladder on $6n+5$ vertices contains $W_{4n}$ as a vertex-minor.
\end{lemma}
\begin{proof}
  Let $G$ be the hanging ladder on $6n+5$ vertices with the labels as in the definition.
  Let $H= G\pivot v_{3}w_3\pivot v_6w_6\pivot\cdots \pivot v_{3n}w_{3n}$.
  Then the vertex set $\{v_i,w_i: 1\le i \le 3n+2, i\not\equiv 0\pmod 3\}$ induces a cycle in $H$.
  As $c$ is still adjacent to $4n$ vertices on this cycle, $H$ is isomorphic to a subdivision of $W_{4n}$.
\end{proof}
\subsection{From extended drums}\label{subsec:extendeddrum}

An \emph{extended drum} of order $n$ 
is a graph $G$ on the vertex set $\{w_1,w_2,\ldots,w_n, u_1, u_2, \ldots, u_n\}\cup S$ such that 
\begin{itemize}
\item $S$ and $\{w_1,w_2,\ldots,w_n, u_1, u_2, \ldots, u_n\}$ are disjoint, 
\item $w_1w_2\cdots w_nw_1$ is an induced cycle, 
\item $\{u_1, u_2, \ldots, u_n\}$ is an independent set,
\item $u_i$ is adjacent to $w_j$ if and only if $i=j$,
\item each $u_i$ has a neighbor in $S$, 
\item there are no edges between $S$ and $\{w_1, \ldots, w_n\}$.
\item there exists a vertex $w\in S$ where 
for every $v\in N_G(\{u_1, u_2, \ldots, u_n\})\cap S$, there is a path $P$ from $v$ to $w$ in $G[S]$ with $N_G(\{u_1, u_2, \ldots, u_n\})\cap V(P)=\{v\}$.
\end{itemize}

\begin{lemma}\label{lem:extendeddrum}
  For an integer $n\ge 3$, 
  an extended drum of order $\mu(n)$ has $W_n$ as a vertex-minor.
\end{lemma}
\begin{proof}
Let $G$ be an extended drum of order $\mu(n)$.
By Lemma~\ref{lem:reduceconnected}, 
there exist $U\subseteq \{u_1, u_2, \ldots, u_{\mu(n)}\}$ with $\abs{U}\ge n$ and $v\in S$ and a graph $G'$ on
$\{w_1, w_2, \ldots, w_{\mu(n)}\}\cup U\cup \{v\}$ such that
\begin{itemize}
\item $G'[\{w_1, w_2, \ldots, w_{\mu(n)}\}\cup U]=G[\{w_1, w_2, \ldots, w_{\mu(n)}\}\cup U]$,
\item $v$ is adjacent to all vertices in $U$ and has no neighbors in $\{w_1, w_2, \ldots, w_{\mu(n)}\}$ in $G'$, 
\item $G'$ is a vertex-minor of $G$.
\end{itemize}
Then $G'$ is a subdivision of $W_n$, and therefore, $G$ contains a vertex-minor isomorphic to $W_n$.
\end{proof}

\subsection{From extended clams}\label{subsec:extendedclam}

An \emph{extended clam} of order $n$ 
is a graph $G$ on the vertex set 
\[\{p_1, p_2, \ldots, p_{2n}, v_1,v_2,\ldots,v_n, w_1, w_2, \ldots, w_n,h\} \cup S\]
such that 
\begin{itemize}
\item $S$ and $\{p_1, p_2, \ldots, p_{2n}, v_1,v_2,\ldots,v_n, w_1, w_2, \ldots, w_n,h\}$ are disjoint, 
\item $p_1p_2\cdots p_{2n}$ is an induced path, 
\item $\{v_1, \ldots, v_n, w_1, \ldots, w_n\}$ is an independent set, 
\item $v_i$ is adjacent to $p_j$ if and only if $j=2i-1$,
\item $w_i$ is adjacent to $p_j$ if and only if $j=2i$, 
\item $h$ is adjacent to all vertices in $\{v_1,\ldots, v_{n}\}$, but non-adjacent to $\{p_1, \ldots, p_{2n}\}\cup S$,  
\item each $w_i$ has a neighbor in $S$, and there are no edges between $S$ and $\{v_1, \ldots, v_n, p_1, \ldots, p_{2n}\}$, 
\item there exists a vertex $w\in S$ where 
for every $v\in N_G(\{w_1, w_2, \ldots, w_n\})\cap S$, there is a path $P$ from $v$ to $w$ in $G[S]$ with $N_G(\{w_1, w_2, \ldots, w_n\})\cap V(P)=\{v\}$.
\end{itemize}
The \emph{simple} extended clam is an extended clam such that $S$ consists of one vertex $z$ that is adjacent to all vertices in $\{w_1, \ldots, w_n\}$ and
$h$ is adjacent to all vertices in $\{w_1, \ldots, w_n\}$. 

\begin{lemma}\label{lem:extendedsimpleclam}
  For an integer $n\ge 2$, the simple extended clam of order $2n+1$ contains a clam on $3n+4$ vertices as a vertex-minor, and thus 
  contains  $W_{2n}$ or $W_{2n+1}$ as a vertex-minor.
\end{lemma}
\begin{proof}
Let $G$ be the simple extended clam of order $2n+1$, and let 
$G_1:=G*w_1*w_2*w_3* \cdots *w_{2n}$.
In $G_1$, both $z$ and $h$ are adjacent to $p_{2i}$ for all $i\in \{1, \ldots, 2n\}$.
Then \[G_1 - \{w_1,w_2,\ldots,w_{2n+1},v_3,v_5,\ldots,v_{2n-1},p_{4n+2}\}\] is a subdivision of a clam on 
$ (4n+2)-(n-1+1) +2= 3n+4$
vertices. By Lemma~\ref{lem:clam}, it contains a vertex-minor isomorphic to $W_{2n}$ or $W_{2n+1}$.
\end{proof}

\begin{lemma}\label{lem:extendedclam}
  For an integer $n\ge 3$, an extended clam of order $\mu(n)+\mu(2n+1)-1$
  contains $W_n$ as a vertex-minor.
\end{lemma}
\begin{proof}
Let $G$ be an extended clam or order $\mu(n)+\mu(2n+1)-1$.
There exists $I\subseteq \{1, \ldots, \mu(n)+\mu(2n+1)-1\}$ such that either 
\begin{itemize}
\item $\abs{I}\ge \mu(n)$ and $h$ is anti-complete to $\{w_i : i\in I\}$, or 
\item $\abs{I}\ge \mu(2n+1)$ and $h$ is complete to $\{w_i : i\in I\}$.
\end{itemize}
When $\abs{I}\ge \mu(n)$ and $h$ is anti-complete to $\{w_i : i\in I\}$, $G$ contains a subdivision of an extended drum of order $\mu(n)$, with the cycle $hv_1p_1p_2 ...p_{(2\mu(n)+2\mu(2n+1)-3)}v_{(\mu(n)+\mu(2n+1)-1)}h$. 
Then by Lemma~\ref{lem:extendeddrum}, it contains a vertex-minor isomorphic to $W_n$.
We may assume that $\abs{I}\ge \mu(2n+1)$ and $h$ is complete to $\{w_i : i\in I\}$. 
Let $G_1:=G-\{v_i,w_i: i\in \{1, 2, \ldots, \mu(n)+\mu(2n+1)-1\}\setminus I\}$, 
and let $A:=V(G_1)\setminus (\{w_i: i\in I\} \cup S)$.

By Lemma~\ref{lem:reduceconnected}, 
there exist $U\subseteq \{w_i : i\in I\}$ with $\abs{U}\ge 2n+1$ and $v\in S$ and a graph $G_2$ on
$A\cup U\cup \{v\}$ such that
\begin{itemize}
\item $G_2[A\cup U]=G[A\cup U]$,
\item $v$ is adjacent to all vertices in $U$ and has no neighbors in $A$ in $G_2$, 
\item $G_2$ is a vertex-minor of $G$.
\end{itemize}
Then $G_2$ is a subdivision of the simple extended clam of order $2n+1$, and therefore, $G$ contains a vertex-minor isomorphic to $W_{2n}$ or $W_{2n+1}$ by Lemma~\ref{lem:extendedsimpleclam}.
\end{proof}

\subsection{From extended hanging ladders}\label{subsec:extendedladder}

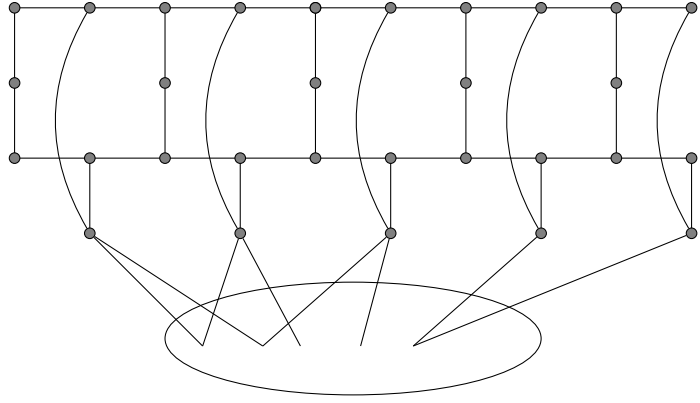
\begin{figure}
  \centering
  \begin{tikzpicture}
    \tikzstyle{every node}=[circle,draw,fill=black!50,inner sep=0pt,minimum width=4pt]

   \draw (1,0)--(10,0);
   \draw (1,2)--(10,2);
    \foreach \y in {1,...,5}{
    
	\draw (2*\y-1,2)--(2*\y-1,1) --(2*\y-1, 0);
	\draw (2*\y,-1) --(2*\y, 0);
      \draw (2*\y-1,2) node (f\y) {};
      \draw (2*\y-1,1) node () {};
      \draw (2*\y,2) node (g\y) {};
      \draw (2*\y,-1) node (e\y) {};
  \draw(e\y) [in=-120,out=120] to (g\y);
    }

    \foreach \y in {1,...,10}{
      \draw (\y,0) node (c\y) {};
    }
      \draw (5,2) node (h) {};

 \draw (e1)--(5.3-1, -2.5);
 \draw (e1)--(5.3-1.8, -2.5);
 \draw (e2)--(5.3-0.5, -2.5);
 \draw (e2)--(5.3-1.8, -2.5);
 \draw (e3)--(5.6, -2.5);
 \draw (e3)--(5.3-1, -2.5);
 \draw (e4)--(5.3+1, -2.5);
 \draw (e5)--(5.3+1, -2.5);
 
 \draw[yscale=0.3] (5.5,-8) circle (2.5);
 \node[draw opacity=0,fill opacity=0,text opacity=1] at (5.3-3, -2.5) {$S$};

  \end{tikzpicture}
  \caption{A simple extended hanging ladder of order $5$.}\label{fig:simplehangingladder}
\end{figure}

For integers $t, n\ge 2$, a  \emph{$t$-extended hanging ladder} of order $n$ is a graph $G$ on the vertex set 
\[\{p_1,p_2,\ldots,p_{2n},v_1, v_2, \ldots, v_n, w_1,w_2,\ldots,w_n\}\cup \{ q_1, q_2, \ldots, q_r\} \cup S\] for some $r$ such that
\begin{itemize}
\item $S$ and $\{p_1,p_2,\ldots,p_{2n},v_1, v_2, \ldots, v_n, w_1,w_2,\ldots,w_n\}\cup \{ q_1, q_2, \ldots, q_r\}$ are disjoint, 
\item $p_1p_2\cdots p_{2n}$, $q_1q_2\cdots q_r$ are induced paths, and $p_i$ is not adjacent to $q_j$,
\item $\{v_1, \ldots, v_n, w_1, \ldots, w_n\}$ is an independent set, 
\item $v_i$ is adjacent to $p_j$ if and only if $j=2i-1$, 
\item $w_i$ is adjacent to $p_j$ if and only if $j=2i$,
\item there exists 
a sequence $1\le b_1< b_2<\cdots< b_n<b_{n+1}=r+1$ such that 
for each $i\in \{1, \ldots, n\}$, 
$v_{i}$ is adjacent to $q_{b_i}$
and non-adjacent to $q_{x}$ for all $x\in \{1, \ldots, r\}\setminus\{b_i,b_i+1,b_i+2,\ldots,b_{i+1}-1\}$,
\item every $w_i$ has a neighbor in $S$, and has at most $t-1$ neighbors on $\{q_1, q_2, \ldots, q_r\}$,
\item there are no edges between $S$ and $\{p_1,p_2,\ldots,p_{2n},v_1, v_2, \ldots, v_n\} \cup \{q_1, q_2, \ldots, q_r\}$, 
\item there exists a vertex $w\in S$ where 
for every $v\in N_G(\{u_1, u_2, \ldots, u_n\})\cap S$, there is a path $P$ from $v$ to $w$ in $G[S]$ with $N_G(\{u_1, u_2, \ldots, u_n\})\cap V(P)=\{v\}$.
\end{itemize}
A \emph{simple} extended hanging ladder is a $t$-extended hanging
ladder for some $t\ge 2$ such that 
\begin{itemize}
\item $r=2n$, 
\item $v_i$ is adjacent to $q_j$ if and only if $j=2i-1$, 
\item $w_i$ is adjacent to $q_j$ if and only if $j=2i$.
\end{itemize} 
Note that the value $t$ is not important in a simple extended hanging ladder because every $w_i$ has exactly one neighbor on $\{q_1, q_2, \ldots, q_r\}$.
We depict a simple hanging ladder in Figure~\ref{fig:simplehangingladder}.

\begin{lemma}\label{lem:extendedhangingladder1}
  For an integer $n\ge 3$, a simple extended hanging ladder of order 
  $\mu(2n+2)$ contains $W_{4n}$ as a vertex-minor.
\end{lemma}
\begin{proof}
Let $G$ be a simple extended hanging ladder of order $\mu(2n+2)$, and let $A:=V(G)\setminus (S\cup \{w_1, w_2, \ldots, w_{\mu(2n+2)}\})$.
By Lemma~\ref{lem:reduceconnected}, 
there exist $U\subseteq \{w_1, w_2, \ldots, w_{\mu(2n+2)}\}$ with $\abs{U}\ge 2n+2$ and $v\in S$ and a graph $G'$ on
$A\cup U\cup \{v\}$ such that
\begin{itemize}
\item $G'[A\cup U]=G[A\cup U]$,
\item $v$ is adjacent to all vertices in $U$ and has no neighbors in $A$ in $G'$, 
\item $G'$ is a vertex-minor of $G$.
\end{itemize}
 Let $U:=\{w_{i_1}, w_{i_2}, \ldots, w_{i_{2n+2}}\}$ where $i_1<i_2<\cdots <i_{2n+2}$.
Then 
\[G_1*w_{i_1}*w_{i_2}* \cdots *w_{i_{2n+1}} *v_{i_1}*v_{i_2}*v_{i_4}*v_{i_6}*\cdots *v_{i_{2n+2}}  \] 
contains an induced subgraph isomorphic to a subdivision of a hanging ladder on $6n+5$ vertices, 
and by Lemma~\ref{lem:hangingladder}, it contains a vertex-minor isomorphic to $W_{4n}$.
\end{proof}

\begin{lemma}\label{lem:extendedhangingladder2}
  For every integer $n\ge 3$, there exists an integer $L=L(n)$ such that an $n$-extended hanging ladder of order  $L$ contains $W_n$ as a vertex-minor.
\end{lemma}
\begin{proof}
Let 
\begin{itemize}
\item $m_1:=\mu(n)$, 
\item $m_2:=8n$, 
\item $m_3:=\mu(2n+2)$, 
\item $m_4:=m_1+2(m_2-1)(n-2)+4$, and 
\item $L:=(m_3-1)m_4+\frac{m_4+m_1}{2}$.
\end{itemize}
Let $G$ be an $n$-extended hanging ladder of order $L$. 
We claim that $G$ contains $W_n$ as a vertex-minor.

We first prove two special cases.

\begin{claim}\label{claim:nonnbrs}
Suppose there exists $i\in \{0, 1, \ldots, L-m_1\}$ such that
there are no edges between $\{w_{i+1}, w_{i+2}, \ldots, w_{i+m_1}\}$ and $\{q_{b_{i+1}}, q_{b_{i+1}+1}, q_{b_{i+1}+2}, \ldots, q_{b_{i+m_1+1}}\}$. 
Then $G$ contains a vertex-minor isomorphic to $W_n$.
\end{claim}
\begin{clproof}
Suppose there exists such an integer $i$. Let $i'$ be the maximum integer such that $v_{i+1}$ is adjacent to $q_{i'}$.
Note that $b_{i+1}\le i'<b_{i+2}$.
Then  
\[q_{i'}q_{i'+1} \cdots q_{b_{i+m_1+1}}v_{i+m_1+1}p_{2(i+m_1+1)-1}p_{2(i+m_1+1)-2} \cdots p_{2i+1}v_{i+1}q_{i'}\]
is an induced cycle.
Since there are no edges between $\{w_{i+j}:1\le j\le m_1\}$ and $\{q_{x}:b_{i+1}\le x\le b_{i+m_1+1}\}$, 
$G$ contains a subdivision of an extended  drum of order $m_1=\mu(n)$.
By Lemma~\ref{lem:extendeddrum}, $G$ contains a vertex-minor isomorphic to $W_n$.
\end{clproof}

\begin{claim}\label{claim:consecutivenbr1}  If there are $i, j_1, j_2\in \{1, \ldots, L\}$ with $j_2-j_1 \ge m_2$ such that
\begin{itemize}
\item $w_i$ is adjacent to $q_{x_1}$ for some $b_{j_1}\le x_1<b_{j_1+1}$ and adjacent to $q_{x_2}$ for some $b_{j_2}\le x_2<b_{j_2+1}$, 
\item $w_i$ is not adjacent to $q_x$ for all $x\in \{x_1+1, x_1+2, \ldots, x_2-1\}$,
\end{itemize}
then $G$ contains a vertex-minor isomorphic to $W_n$. 
\end{claim}

\begin{figure}
  \centering
  \begin{tikzpicture}[scale=0.55]
  \tikzstyle{w}=[circle,draw,fill=black!50,inner sep=0pt,minimum width=4pt]

   \draw (0,0)--(13,0);

 \foreach \z in {1,...,3}{
   \draw (\z*4-1,3)--(\z*4-1,4.5);
   \draw (\z*4-1+1,0)--(\z*4-1,3)--(\z*4-1-1,0);
      \draw (\z*4-1,3) node [w] (z\z) {};
     }
     \draw(3, 4.5)--(11, 4.5);
     
     \node at (4-1+1, 3) {$v_{j_1+2}$};
     \node at (8-1+1, 3) {$v_{j_1+4}$};
     \node at (12-1+1, 3) {$v_{j_1+6}$};

  	\node at (3, 5) {$p_{2j_1+3}$};
     \node at (11, 5) {$p_{2j_1+11}$};

        \node at (8-1+1, -3) {$w_i$};
           \draw (8-1,-3) node [w] (x) {};

               \node at (2, -.5) {$q_{b_{j_1+2}}$};
               \node at (6, -.5) {$q_{b_{j_1+4}}$};
               \node at (10, -.5) {$q_{b_{j_1+6}}$};
     
   \draw (1,0) node [w] (a1) {};
   \draw (13,0) node [w] (a13) {};
 \foreach \y in {0,...,12}{
      \draw (\y,0) node [w] (a\y) {};
     }
 \foreach \y in {3,...,11}{
      \draw (\y,4.5) node [w] (b\y) {};
     }
	\draw(0,0)--(x)--(a13);

	\draw[thick, blue, dashed](0.8, 0.5)--(0.8, 1)--(4.2,1)--(4.2,0.5);
	\draw[thick, blue, dashed](4.8, 0.5)--(4.8, 1)--(8.2,1)--(8.2,0.5);
	\draw[thick, blue, dashed](8.8, 0.5)--(8.8, 1)--(12.2,1)--(12.2,0.5);

   \end{tikzpicture}\qquad
     \begin{tikzpicture}[scale=0.55]
  \tikzstyle{w}=[circle,draw,fill=black!50,inner sep=0pt,minimum width=4pt]

   \draw (0,0)--(13,0);

 \foreach \z in {1,...,3}{
   \draw (\z*4-1,3)--(\z*4-1,4.5);
   \draw (\z*4-1,0)--(\z*4-1,3);
      \draw (\z*4-1,3) node [w] (z\z) {};
     }
     \draw(3, 4.5)--(11, 4.5);
     
     \node at (4-1+1, 3) {$v_{j_1+2}$};
     \node at (8-1+1, 3) {$v_{j_1+4}$};
     \node at (12-1+1, 3) {$v_{j_1+6}$};

        \node at (8-1+1, -3) {$w_i$};
           \draw (8-1,-3) node [w] (x) {};
   
          	\node at (3, 5) {$p_{2j_1+3}$};
     \node at (11, 5) {$p_{2j_1+11}$};

  \foreach \y in {0,3,7,11, 13}{
      \draw (\y,0) node [w] (a\y) {};
     }
 \foreach \y in {3,...,11}{
      \draw (\y,4.5) node [w] (b\y) {};
     }
	\draw(0,0)--(x)--(a13);

   \end{tikzpicture}

     \caption{An example of contractions in Claim~\ref{claim:consecutivenbr1} of Lemma~\ref{lem:extendedhangingladder2}. We will contract dashed parts in the first figure to obtain the second figure.}\label{fig:examplecontractionlem}
\end{figure}
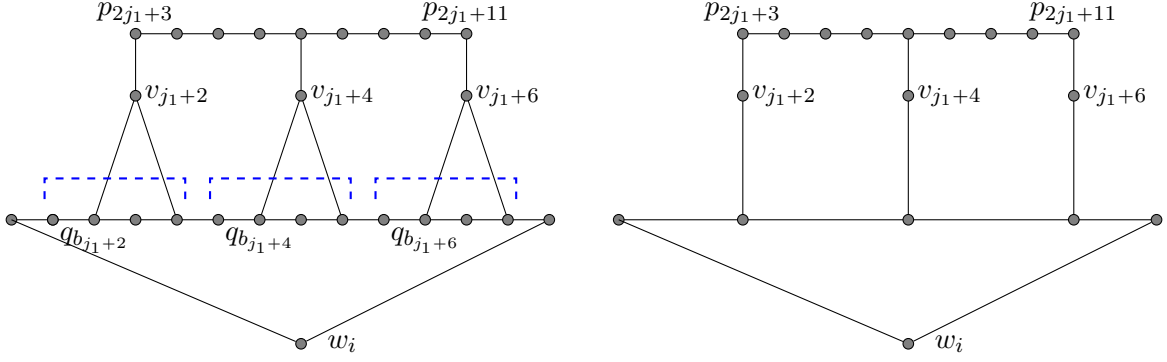

\begin{clproof}
Suppose there are such integers $i, j_1, j_2$. Then 
$ w_iq_{x_1}q_{x_1+1}q_{x_1+2} \cdots q_{x_2}w_i$
is an induced cycle. 

First assume that $i\le j_1$.
Let 
\[G_1:=G[\{w_i, q_{x_1}, q_{x_1+1},  \ldots,  q_{x_2}\}\cup \{v_{j_1+2}, v_{j_1+4}, v_{j_1+6}, \ldots, v_{j_2'}\} \cup \{p_{2j_1+3}, p_{2j_1+4},  \ldots, p_{2j_2'-1}\}],\]
where $j_2'=j_2-2$ if $j_2\equiv j_1 \pmod{2}$ and $j_2'=j_2-1$ otherwise.
We will contract paths from $G$ to obtain a drum on $\frac{3}{2}(m_2-2)\ge 3(2n-1)$ vertices.
See Figure~\ref{fig:examplecontractionlem} for an example case.
Observe that $p_{2j_1+3}p_{2j_1+4} \cdots p_{2j_2'-1}$ is a path such that each vertex of $v_{j_1+2}, v_{j_1+4}, v_{j_1+6}, \ldots, v_{j_2'}$ has a neighbor on this path.

Let $G_2$ be the graph obtained from $G_1$ by contracting 
\[\{q_{b_{j_1+(t-1)}}, q_{b_{j_1+(t-1)}+1}, \ldots,  q_{b_{j_1+(t+1)}-1}\}\] for each 
$t\in \{2, 4, 6, \ldots, j_2'-j_1\}$. 
By Lemma~\ref{lem:fancontract}, $G_2$ is isomorphic to a vertex-minor of $G_1$.
Moreover, $G_2$ contains a subdivision of a drum on $\frac{3}{2}(j_2-j_1-2)=\frac{3}{2}(m_2-2)\ge 3(2n-1)$ vertices, and 
by Lemma~\ref{lem:drum}, $G_2$ contains a vertex-minor isomorphic to $W_n$.
The case when $i\ge j_2$ is symmetric to the previous case. We may assume that $j_1<i<j_2$.
In this case, $w_i$ is adjacent to $p_{2i}$, and to avoid having this edge, we take a part that is larger and having no edges from $w_i$.

Since $j_2-j_1\ge m_2$, we have either $i-j_1\ge m_2/2$ or $j_2-i\ge m_2/2$.
So, by taking the longer path between two subpaths obtained from $p_{2j_1+3}p_{2j_1+4}\cdots p_{2j_2'-1}$ by removing $p_{2i}$, we can observe that $G$ contains a vertex-minor isomorphic to a drum on $\frac{3}{2}(\frac{m_2}{2}-2)$ vertices. 
Since $\frac{3}{2}(\frac{m_2}{2}-2)=3(2n-1)$, by Lemma~\ref{lem:drum}, $G$ contains a vertex-minor isomorphic to $W_n$.
\end{clproof}

From Claim~\ref{claim:consecutivenbr1}, 
we observe the following.

\begin{claim}\label{claim:consecutivenbr2}  If there are $i, j_1, j_2\in \{1, \ldots, L\}$ with $j_2-j_1 \ge (m_2-1)(n-2)+1$ such that
$w_i$ is adjacent to $q_{x_1}$ for some $b_{j_1}\le x_1<b_{j_1+1}$ and adjacent to $q_{x_2}$ for some $b_{j_2}\le x_2<b_{j_2+1}$, 
then $G$ contains a vertex-minor isomorphic to $W_n$. 
\end{claim}
\begin{clproof}
Since $j_2-j_1\ge (m_2-1)(n-2)+1$ and $w_i$ has at most $n-1$ neighbors in $\{q_1, q_2, \ldots, q_r\}$, there exist $j_3, j_4$ with $j_1<j_3<j_4<j_2$ such that 
\begin{itemize}
\item $j_4-j_3\ge m_2$, 
\item $w_i$ is adjacent to $q_{x_3}$ for some $b_{j_3}\le x_3<b_{j_3+1}$ and adjacent to $q_{x_4}$ for some $b_{j_4}\le x_4<b_{j_4+1}$, 
\item $w_i$ is not adjacent to $q_x$ for all $x_3<x<x_4$.
\end{itemize}
By Claim~\ref{claim:consecutivenbr1}, $G$ contains a vertex-minor isomorphic to $W_n$.
\end{clproof}

For each $i\in \{1, 2, \ldots, m_3\}$, we choose a vertex $w_{d_i}$ such that
\begin{itemize}
\item $d_i\in \{ im_4+1, im_4+2, \ldots, im_4+m_1\}$ and
\item $w_{d_i}$ is adjacent to $q_x$ for some $b_{im_4+1}\le x\le b_{im_4+m_1+1}$. 
\end{itemize}
If such a vertex does not exist for some $i$, then by Claim~\ref{claim:nonnbrs}, $G$ contains a vertex-minor isomorphic to $W_n$. 
We may assume that such a vertex exists for each $i\in \{1, 2, \ldots, m_3\}$.

Suppose $w_{d_i}$ is adjacent to $q_y$ for some $y\ge b_{im_4+\frac{m_4+m_1}{2}}$ or for some $y\le b_{im_4-\frac{m_4-m_1}{2}+3}-1$.
By the choice of $d_i$, $w_{d_i}$ is adjacent to $q_x$ for some $b_{im_4+1}\le x\le b_{im_4+m_1+1}$.
Since $\frac{m_4+m_1}{2}-(m_1+1)\ge (m_1-1)(n-2)+1$ and $1- (-\frac{m_4-m_1}{2}+2)\ge (m_1-1)(n-2)+1$, 
by Claim~\ref{claim:consecutivenbr2}, $G$ contains a vertex-minor isomorphic to $W_n$. 
We may assume that each $w_{d_i}$ is not adjacent to $q_y$ for all $y\in \{1,2, \ldots, r\}\setminus \{j: b_{im_4-\frac{m_4-m_1}{2}+3}\le j< b_{im_4+\frac{m_4+m_1}{2}}  \}$.
Note that $(i+1)m_4-\frac{m_4-m_1}{2}=im_4+\frac{m_4+m_1}{2}$.

\medskip 
Let $G_1$ be the subgraph of $G$ induced on 
\[ \{w_{d_i}: 1\le i\le m_3\} \cup \{v_{im_4-\frac{m_4-m_1}{2}+1}: 1\le i\le m_3 \} \cup \{p_1, \ldots, p_{2L}, q_1, \ldots, q_r\}\cup S. \] 
Let $G_2$ be the graph obtained from $G_1$ by contracting 
\[ \{q_{b_{im_4-\frac{m_4-m_1}{2}+2}}, q_{b_{im_4-\frac{m_4-m_1}{2}+2}+1}, \ldots, q_{b_{im_4+\frac{m_4+m_1}{2}}-1}\} \]
for each $i\in \{1, 2, \ldots, m_3\}$ and contracting
\[ \{q_{b_{im_4-\frac{m_4-m_1}{2}}}, q_{b_{im_4-\frac{m_4-m_1}{2}}+1}, \ldots, q_{b_{im_4-\frac{m_4-m_1}{2}+2}-1}\}  \]
for each $i\in \{1, 2, \ldots, m_3\}$.
By Lemma~\ref{lem:fancontract}, $G_2$ is isomorphic to a vertex-minor of $G$.
Also, $G_2$ contains a subdivision of a simple extended hanging ladder of order $m_3=\mu(2n+2)$.
By Lemma~\ref{lem:extendedhangingladder1}, $G_2$ contains a vertex-minor isomorphic to $W_{4n}$.
\end{proof}

\section{$(w, \ell)$-patched cycles}\label{sec:patchedpath}
Let $w, \ell$ be positive integers. A \emph{$(w, \ell)$-patched cycle} $(q_1q_2 \cdots q_mq_1, S_1, S_2, \ldots, S_w)$ is a graph $G$ on pairwise disjoint sets 
$\{q_1, q_2, \ldots, q_m\}$ and $S_i=\{s^i_1, s^i_2, \ldots, s^i_\ell\}$ for each $i\in \{1, \ldots, w\}$ 
satisfying the following.
\begin{enumerate}[(1)]
\item $q_1q_2 \cdots q_mq_1$ is an induced cycle.
\item There exists 
a sequence $1\le b_1< b_2<\cdots< b_{\ell}\le m$ such that 
for each $i\in \{1, \ldots, w\}$ and $j\in \{1,2,\ldots, \ell\}$, 
$s^i_j$ is adjacent to $q_{b_j}$
and non-adjacent to $q_{b_x}$ for all $x\in \{1, \ldots, \ell\}\setminus \{1, \ldots, j\}$.
\end{enumerate}
We call $w$ and $\ell$ the \emph{width} and \emph{length} respectively, of a $(w, \ell)$-patched cycle.
Note that $m\ge \ell$.
A $(w, \ell)$-patched cycle $(q_1q_2 \cdots q_mq_1, S_1, S_2, \ldots, S_w)$ is \emph{simple} if $S_1\cup S_2\cup \cdots \cup S_w$ is an independent set.

In Subsection~\ref{subsec:fromtwoellpatchedpath}, we show that 
for every $n$, there exists $M$ such that every graph $G$ obtained from a simple $(2, M)$-patched cycle $(q_1q_2 \cdots q_mq_1, S_1, S_2)$ by 
adding two disjoint vertex sets $T_1$ and $T_2$ such that 
\begin{itemize}
\item there are no edges between $\{q_1, q_2, \ldots, q_m\}$ and $T_1\cup T_2$, 
\item there are no edges between $S_1$ and $T_2$,
\item for each $i\in \{1,2\}$, every vertex in $S_i$ has a neighbor in $T_i$, 
\item for each $i\in \{1,2\}$, 
there exists a vertex $r_i\in T_i$ where 
for every $v\in N_G(S_i)\cap T_i$, there is a path $P$ from $v$ to $r_i$ in $G[T_i]$ with $N_G(S_i)\cap V(P)=\{v\}$,
\end{itemize}
contains a vertex-minor isomorphic to $W_n$. This is the motivation for introducing $(w, \ell)$-patched cycles.
In Subsection~\ref{subsec:obtainsimplepatchedpath}, 
we show that for every $n$, a simple $(2, n)$-patched cycle can be obtained from  a patched cycle with sufficiently large width and large length using the rectangular Ramsey lemma developed in Section~\ref{sec:ramsey}.
In Subsection~\ref{subsec:obtainhugepatchedpath}, 
we discuss how to obtain a huge patched cycle from a structure that can be naturally extracted from a graph with bounded clique number and sufficiently large chromatic number.

\subsection{From a simple $(2, M)$-patched cycle with attached connected subgraphs}\label{subsec:fromtwoellpatchedpath}

\begin{proposition}\label{prop:finalconfig2}
For every integer $n\ge3$, there exists an integer $M=M(n)$
satisfying the following property:
If $G$ is the graph obtained from a simple $(2, M)$-patched cycle $(q_1q_2 \cdots q_mq_1, S_1, S_2)$ by adding disjoint vertex sets $T_1$ and $T_2$ such that 
\begin{itemize}
\item there are no edges between $\{q_1, q_2,  \ldots,  q_m\}$ and $T_1\cup T_2$,
\item there are no edges between $S_1$ and $T_2$,
\item for each $i\in \{1,2\}$, every vertex in $S_i$ has a neighbor in $T_i$,
\item for each $i\in \{1,2\}$, 
there exists a vertex $r_i\in T_i$ where 
for every $v\in N_G(S_i)\cap T_i$, there is a path $P$ from $v$ to $r_i$ in $G[T_i]$ with $N_G(S_i)\cap V(P)=\{v\}$,
\end{itemize}
then $G$ contains a vertex-minor isomorphic to $W_n$.
\end{proposition}
\begin{proof}
We recall that $\mu(n)=(n-1)(R(n,n)^{2n-3}+1)$ for $n\ge 3$, $N'$ is the function defined in Corollary~\ref{cor:regularpartition}, and $L$ is the function defined in Lemma~\ref{lem:extendedhangingladder2}.
We note that $L(n)\ge \mu(n)+\mu(2n+1)-1$.
Let 
\begin{itemize}
\item $M_1:=(n-1)(4L(n)+6)$,
\item $M_2:=(n-1)N'(n-1, M_1)$, 
\item $M:=N'(n-1, M_2)$.
\end{itemize}
For each $i\in \{1, 2\}$, let $S_i:=\{s^i_1, \ldots, s^i_M\}$, and
let $b_1, \ldots, b_{M}$ be a sequence such that 
\begin{itemize}
\item $1\le b_1< b_2<\cdots< b_{M}\le m$ and
\item for each $i\in \{1, 2\}$ and $j\in \{1,2,\ldots, M\}$, 
$s^i_j$ is adjacent to $q_{b_j}$
and non-adjacent to $q_{b_x}$ for all $x\in \{1, \ldots, M\}\setminus \{1, \ldots, j\}$.
\end{itemize}
Such a sequence exists by the definition of a $(2, M)$-patched cycle.
Let $Q:=\{q_1, q_2, \ldots, q_m\}$, and 
for each $i\in \{1,2\}$ and $j\in \{1, \ldots, M\}$, let $N^i_j=\{k:
q_k\in N_G(s^i_j), 1\le k\le M\}$.

Note that $m\ge M\ge n+3$. If there is a vertex in $S_1\cup S_2$ having at least $n$ neighbors on
$Q$, then $G$ contains a vertex-minor isomorphic to $W_n$ by Lemma~\ref{lem:fromlargerwheel}.
We may assume that each vertex in $S_1\cup S_2$ has at most $n-1$ neighbors in $Q$.
In other words, for each $i\in \{1,2\}$ and $j\in \{1, \ldots, M\}$, $\abs{N^i_j}\le n-1$.

We apply Corollary~\ref{cor:regularpartition} to $(N^1_1, \ldots, N^1_{M})$.
Then 
there exist a sequence $1\le c_1<c_2<\cdots<c_{M_2} \le M$ and a
regular partition $\mathcal{I}_1$ of $\mathbb R$ with respect to 
$(N^1_{c_1}, N^1_{c_2}, \ldots, N^1_{c_{M_2}})$
such that $\mathcal{I}_1$ has order at most $n-1$.
Since $\mathcal{I}_1$ has order at most $n-1$, there exists a part $I_1$ of
$\mathcal{I}_1$ that contains at least $\frac{M_2}{n-1}$ integers in
$\{b_{c_1}, b_{c_2}, \ldots,b_{c_{M_2}}\}$.
Let $x$ be the minimum such that $x\ge 1$ and $b_{c_x}\in I_1$,
and let $y$ be the maximum such that $y\le M_2$ and $b_{c_y}\in I_1$.
Then $y-x+1\ge \frac{M_2}{n-1}=N'(n-1, M_1)$.

We apply Corollary~\ref{cor:regularpartition} again to $(N^2_1\cap I_1, \ldots, N^2_{M}\cap I_1)$.
Then 
there exist a subsequence $d_1<d_2<\cdots<d_{M_1} $ of $c_x,c_{x+1},\ldots,c_{y}$ and a
regular partition $\mathcal{I}_2$ of $I_1$ with respect to 
$(N^2_{d_1}\cap I_1, N^2_{d_2}\cap I_1, \ldots, N^2_{d_{M_1}}\cap I_1)$
such that $\mathcal{I}_2$ has order at most $n-1$. 
There exists
a part $I_2$ of $\mathcal{I}_2$ 
that contains at least $\frac{M_1}{n-1}$ integers in 
$\{b_{d_1},b_{d_2},b_{d_3},\ldots,b_{d_{M_1}}\}$.
Let $x'$ be the minimum such that $x'\ge 1$ and $b_{d_{x'}}\in I_2$
and 
let $y'$ be the maximum such that $y'\le M_1$ and $b_{d_{y'}}\in I_2$.
Let $a=y'-x'+1$. 
Then $a=y'-x'+1\ge \frac{M_1}{n-1}=4L(n)+6$.
Let $u_1=d_{x'}$, $u_2=d_{x'+1}$, $\ldots$, $u_{a}=d_{y'}$.

By the definition of a $(2, M)$-patched cycle,  for each $i\in \{1,2\}$
and $1\le j\le a$, 
$s^i_{u_j}$ is adjacent to $q_{b_{u_j}}$ but non-adjacent to vertices in $\{q_{b_{u_{j+1}}},q_{b_{u_{j+2}}},\ldots,q_{b_{u_{a}}}\}$.
Therefore, $N^i_{u_j}\cap I_2$ is not identical for $1\le j\le a$ and moreover, minimal intervals containing $N^i_{u_1}\cap
I_2, N^i_{u_2}\cap I_2, \ldots, N^i_{u_a}\cap I_2$ appear in the same order as $u_1, u_2, \ldots, u_a$. 
In other words, for each $i\in \{1,2\}$,
\begin{itemize}
\item $N^i_{u_1}\subseteq (-\infty,b_{u_{2}})\cap I_2$, 
\item $N^i_{u_a}\subseteq (b_{u_{a-1}},\infty)\cap I_2$, 
\item for $j\in \{2,3,\ldots,a-1\}$, 
   $N^i_{u_j}\subseteq (b_{u_{j-1}},b_{u_{j+1}})$.
\end{itemize}
Let $t=4L(n)+4$.

We first deal with the case when $\mathcal{I}_1$ consists of one part.

\begin{claim}\label{claim:onepart}
  If $\mathcal{I}_1$ consists of one part, then $G$ contains a vertex-minor isomorphic to $W_n$.
\end{claim}
\begin{clproof}
Let $G_1$ be the subgraph of $G$ induced on $Q \cup \{s^1_{u_2},s^1_{u_4},s^1_{u_6},\ldots,s^1_{u_t}\}\cup T_1 $.
We obtain a graph $G_2$ from $G_1$ by contracting $\{q_x: b_{u_{j-1}}\le x<b_{u_{j+1}}\}$ for each $j\in\{2,4,\ldots,t\}$ to a vertex. 
By Lemma~\ref{lem:fancontract}, $G_2$ is isomorphic to a vertex-minor of $G_1$. 
Then $G_2$  is a
subdivision of an extended drum of order $t/2\ge L(n)\ge \mu(n)$.
By Lemma~\ref{lem:extendeddrum}, 
$G$ contains a vertex-minor isomorphic to $W_n$.
\end{clproof}

By Claim~\ref{claim:onepart}, we may assume that $\mathcal{I}_1$ consists of at least two parts.
Let $J$ be a part of $\mathcal{I}_1$ other than $I_1$. Clearly, $J$ is disjoint from $I_2$.
By the definition of a regular partition,  either 
\begin{enumerate}[(1)]
\item $N^1_{u_2}\cap J=N^1_{u_4}\cap J=\cdots =N^1_{u_{t}}\cap J\neq \emptyset$, or
\item $\abs{N^1_{u_2}\cap J}=\abs{N^1_{u_4}\cap J}=\cdots =\abs{N^1_{u_{t}}\cap J}>0$ and for all $i,j\in \{2,4, \ldots, t \}$ with $i<j$, $\max(N^1_{u_i}\cap J)<\min (N^1_{u_j}\cap J)$, or
\item $\abs{N^1_{u_2}\cap J}=\abs{N^1_{u_4}\cap J}=\cdots =\abs{N^1_{u_{t}}\cap J}>0$ and for all $i,j\in \{2,4, \ldots, t \}$ with $i<j$, $\max (N^1_{u_j}\cap J)< \min (N^1_{u_i}\cap J)$. 
\end{enumerate}
When (1) appears, we will find an extended clam of large order, and when (2) or (3) appears, we will find an extended hanging ladder of large order.

\medskip
\noindent\textbf{Case 1.} $N^1_{u_2}\cap J=N^1_{u_4}\cap J=\cdots =N^1_{u_{t}}\cap J\neq \emptyset$.

Let $w\in N^1_{u_2}\cap J$. 
Let $Q_1$ be the connected component of $G[Q]-q_{b_{u_1}}-q_{b_{u_a}}$ containing
$q_{b_{u_2}}$.
We observe that $q_w\notin V(Q_1)$ and
$q_{b_{u_4}}, q_{b_{u_6}}, \ldots, q_{b_{u_{t-2}}}\in V(Q_1)$.
Let $G_1$ be the subgraph of $G$ induced on
\begin{align*}
\{q_w\}&\cup  V(Q_1)\cup T_2 \\
&\cup \{s^1_{u_j}: 4\le j\le t-2, j\equiv 0 \pmod 4 \} \cup
  \{s^2_{u_j}: 4\le j\le t-2, j\equiv 2 \pmod 4\}.  
\end{align*}
We obtain a graph $G_2$ from $G_1$ by contracting $\{q_x:
b_{u_{j-1}}\le x<b_{u_{j+1}}\}$ for each  $j\in\{4,6,8,\ldots,t-2\}$ to a vertex. 
By Lemma~\ref{lem:fancontract}, $G_2$ is isomorphic to a vertex-minor of $G_1$.
Note that there are no edges between $T_2$ and $S_1$.
Thus, $G_2$ contains a subdivision of an extended clam of order $t/4 -1=L(n)\ge \mu(n)+\mu(2n+1)-1$, and by Lemma~\ref{lem:extendedclam}, $G$ contains a vertex-minor isomorphic to $W_n$.

\medskip
\noindent\textbf{Case 2.} $\abs{N^1_{u_2}\cap J}=\abs{N^1_{u_4}\cap J}=\cdots =\abs{N^1_{u_{t}}\cap J}>0$ and for all $i,j\in \{2,4, \ldots, t \}$ with $i<j$, $\max(N^1_{u_i}\cap J)<\min (N^1_{u_j}\cap J)$.

Let $Q_1$ be the connected component of $G[Q]-q_{b_{u_1}}-q_{b_{u_a}}$ containing
$q_{b_{u_2}}$.
Let $Q_2$ be the path on $\{q_i:i\in J\}$.
Note that there are no edges between $Q_1$ and $Q_2$.
Let $G_1$ be the subgraph of $G$ induced on
\begin{align*}
V(Q_1)&\cup  V(Q_2)\cup T_2 \\
&\cup \{s^1_{u_j}: 4\le j\le t-2, j\equiv 0 \pmod 4 \} \cup
  \{s^2_{u_j}: 4\le j\le t-2, j\equiv 2 \pmod 4\}.  
\end{align*}
We obtain a graph $G_2$ from $G_1$ by contracting $\{q_x:
b_{u_{j-1}}\le x<b_{u_{j+1}}\}$ for each  $j\in\{4,6,8,\ldots,t-2\}$. 
By Lemma~\ref{lem:fancontract}, $G_2$ is isomorphic to a vertex-minor of $G_1$.
Note that there are no edges between $T_2$ and $S_1$.
Thus, $G_2$ contains a subdivision of an $n$-extended hanging ladder of order $t/4-1\ge L(n)$, and by Lemma~\ref{lem:extendedhangingladder2}, $G$ contains a vertex-minor isomorphic to $W_n$.

\medskip
\noindent\textbf{Case 3.} For all $i,j\in \{2,4, \ldots, t \}$ with $i<j$, $\abs{N^1_{u_i}\cap J}=\abs{N^1_{u_j}\cap J}>0$ and  $\max (N^1_{u_j}\cap J)< \min (N^1_{u_i}\cap J)$.

This case is symmetric to \textbf{Case 2}.

\medskip
This completes the proof of the proposition.
 \end{proof}

\subsection{Obtaining a simple patched cycle}\label{subsec:obtainsimplepatchedpath}

\begin{proposition}\label{prop:intermediatesimple}
Let $R_1$, $R_2$ be the functions defined in Proposition~\ref{prop:recramsey}.
For all positive integers $a$, $b$, and $k$, 
if  $M=R_1(a,b,k)$ and $N=R_2(a,b,k)$, then 
every $(M, N)$-patched cycle $(q_1q_2 \cdots q_mq_1, S_1, S_2, \ldots, S_{M})$ contains either a clique of size $k$ or a simple $(a, b)$-patched cycle $(q_1q_2 \cdots q_mq_1, T_1, T_2, \ldots, T_a)$ where $T_1, \ldots, T_a$ are contained in pairwise distinct sets of $S_1, \ldots, S_{M}$.
\end{proposition}
\begin{proof}
For each $i\in \{1, \ldots, M\}$, let $S_i:=\{s^i_1, s^i_2, \ldots, s^i_N\}$ and let $1\le b_1<b_2< \cdots <b_N\le m$ be a sequence such that
\begin{itemize}
\item for each $i\in \{1, \ldots, M\}$ and each $j\in \{1, \ldots, N\}$, 
$s^i_j$ is adjacent to $q_{b_j}$ and non-adjacent to $q_{b_x}$ for all $x>j$.
\end{itemize}
By Proposition~\ref{prop:recramsey}, either $G$ has a clique of $k$ vertices or
 there exist $X\subseteq \{1, 2, \ldots, M\}$ and $Y\subseteq \{1, 2, \ldots, N\}$ such that
 $\{s^i_j:i\in X, j\in Y\}$ is an independent set and $\abs{X}=a$, $\abs{Y}=b$.
 In the latter case, let $X=\{x_1,x_2,\ldots,x_a\}$
and $T_i:=\{s^{x_i}_j:j\in Y\}$ for each $i=1,2,\ldots,a$.
 It is easy to verify that $(q_1q_2 \cdots q_mq_1,T_1,T_2,\ldots,T_a)$ is a simple $(a,b)$-patched cycle.
\end{proof}

\subsection{Obtaining a patched cycle with large width and length}\label{subsec:obtainhugepatchedpath}

We prove the following.

\begin{proposition}\label{prop:intermediate}
Let $k>0$, $\ell>0$, $n\ge 2$ be  integers and let $M:=\ell n^k$.
Let $G$ be a graph on the vertex set $\{q_1, q_2, \ldots, q_{M}\}\cup V_1\cup V_2\cup \ldots \cup V_k$ such that
\begin{itemize}
\item $\{q_1, q_2, \ldots, q_{M}\}$, $V_1$, $V_2$, $\ldots$, $V_k$ are pairwise disjoint,
\item $q_1q_2q_3 \cdots q_{M}q_1$ is an induced cycle, 
\item for each $i\in \{1, 2, \ldots, M\}$ and each $j\in \{1, \ldots, k\}$, $q_i$ has a neighbor in $V_j$,
\item for each vertex $v\in V(G)\setminus \{q_1, q_2, \ldots, q_M\}$, $v$ has at most $n-1$ neighbors in $\{q_1, q_2, \ldots, q_M\}$. 
\end{itemize}
Then $G$ contains a $(k, \ell)$-patched cycle $(q_1q_2 \cdots q_Mq_1, S_1, S_2, \ldots, S_k)$ such that
$S_i\subseteq V_i$ for each $i\in \{1, \ldots, k\}$.
\end{proposition}
\begin{proof}
We prove the statement by induction on $k$.

First assume that $k=1$.
Let  $s_1$  be a neighbor of $q_1$ in $V_1$, and let $b_1:=1$. 
Let $i$ be the maximum integer satisfying the following: there exist distinct vertices $s_1,s_2,\ldots,s_i$ of $V_1$ 
and a sequence $b_1<b_2<\cdots<b_i$ where for all $x\in \{1, \ldots, i\}$, 
$s_x$ is adjacent to $q_{b_x}$
and when $x>1$, 
\begin{itemize}
\item $b_x$ is the minimum integer such that $b_x>b_{x-1}$ and $q_{b_x}$ has no neighbors in $\{s_1, \ldots, s_{x-1}\}$. 
\end{itemize}
Such $i$ exists, because $i=1$ satisfies the conditions.
Suppose that $i<\ell$. 
Note that every vertex $q_j$ for $1\le j\le b_i$ has a neighbor in
$\{s_1, \ldots, s_i\}$, otherwise, let $j'$ be the smallest integer such that
$b_{j'}>j$ and we may replace $b_{j'}$ with $j$, contradicting our assumption
on $b_{j'}$.
Therefore, vertices in $\{s_1, \ldots, s_i\}$ may have at most $(n-1)i-b_i$ neighbors $q_j$ for $j>b_i$.
It implies that there exists $j$ with $b_i<j\le (n-1)i+1\le \ell n$
such that $q_j$ has no neighbors in $\{s_1, \ldots, s_i\}$.
So, we can extend the sequence by taking  $b_{i+1}:=j$ and a neighbor
of $q_{b_{i+1}}$ in $V_1$ as $s_{i+1}$, contradicting to the maximality of $i$. Thus, we have $i\ge \ell$.
Note that by the choice of $b_1,\ldots, b_i$, this sequence satisfies the property that 
\begin{itemize}
\item for each $x\in \{1, \ldots, i\}$, $s_x$ is adjacent to $q_{b_x}$ and non-adjacent to $q_{b_y}$ for all $y>x$.
\end{itemize}
We conclude $G$ contains a $(1, \ell)$-patched cycle $(q_1q_2 \cdots q_Mq_1, S_1)$
with $S_1\subseteq V_1$.

Now, suppose $k>1$.
By the induction hypothesis, $G$ contains a $(k-1, \ell n)$-patched cycle $(q_1q_2 \cdots q_Mq_1, T_1, \ldots, T_{k-1})$ such that $T_i\subseteq V_i$ for each $i\in \{1, \ldots, k-1\}$.
Let $T_i=\{t^i_1, t^i_2, \ldots, t^i_{\ell n}\}$ for each $i\in \{1, \ldots, k-1\}$ and  
let $1\le b_1< b_2<\cdots< b_{\ell n}\le M$ be the sequence such that 
\begin{itemize}
\item for each $i\in \{1, \ldots, k-1\}$ and $j\in \{1,2,\ldots, \ell n\}$, 
$t^i_j$ is adjacent to $q_{b_j}$
and non-adjacent to $q_{b_x}$ for all $x\in \{j+1, \ldots, \ell n\}$.
\end{itemize}
For each $i\in \{1, \ldots, k-1\}$, let $f_i:\{b_1, \ldots, b_{\ell (n-1)+1}\}\rightarrow T_i$ be the bijection such that $f_i(b_j)=t^i_j$.

Let  $s^k_1\in V_k$  be a neighbor of $q_{b_1}$, and let $c_1:=1$. 
Let $i$ be the maximum integer satisfying the following: there exist distinct vertices $s^k_1,s^k_2,\ldots,s^k_i$ of $V_k$ 
and a sequence $c_1<c_2<\cdots<c_i$ where 
for all $x\in \{1, \ldots, i\}$, $q_{b_{c_x}}$ is adjacent to $s^k_x$, and when $x>1$, 
\begin{itemize}
\item 
$c_x$ is the minimum integer such that $c_x>c_{x-1}$ and $q_{b_{c_x}}$ has no neighbors in $\{s^k_1, \ldots, s^k_{x-1}\}$, 
\end{itemize}
Such $i$ exists, because $i=1$ satisfies the conditions.
Suppose that $i<\ell$. 
Note that every vertex $q_{b_j}$ in $\{q_{b_1}, q_{b_2}, \ldots,
q_{b_{{c_i}}}\}$ has a neighbor in $\{s^k_1, \ldots, s^k_i\}$,
otherwise, let $j'$ be the smallest integer such that $c_{j'}>j$ and
we may replace $c_{j'}$ by $j$, contradicting our assumption on $c_{j'}$.
Therefore, vertices in $\{s^k_1, \ldots, s^k_i\}$ may have at most
$(n-1)i-c_i$ neighbors
in  $\{ q_{b_j}: c_i<j\le \ell n\}$.
It implies that there exists $j$ with $c_i<j\le (n-1)i+1\le \ell n$
such that $q_{b_{j}}$ has no neighbors in $\{s^k_1, \ldots, s^k_i\}$.
So, we can extend the sequence by taking  $c_{i+1}:=j$ and a neighbor
of $q_{c_{i+1}}$ in $V_k$ as $s^k_{i+1}$, contradicting to the maximality of $i$. Thus, we have $i\ge \ell$.
Note that by the choice of $c_1,\ldots, c_i$, this sequence satisfies the property that 
\begin{itemize}
\item for each $x\in \{1, \ldots, i\}$, $s^k_x$ is adjacent to $q_{b_{c_x}}$ and non-adjacent to $q_{b_{c_y}}$ for all $y>x$.
\end{itemize}
For each $i\in \{1, \ldots, k-1\}$ and $j\in \{1, \ldots, \ell\}$, let $s^i_j$ be the vertex $f_i(c_j)$.
Then 
\[(q_1q_2 \cdots q_Mq_1, \{s^1_1, \ldots, s^1_{\ell}\}, \ldots, \{s^k_1, \ldots, s^k_{\ell}\})\] is a $(k, \ell)$-patched cycle
such that  $\{s^i_1, \ldots, s^i_{\ell}\}\subseteq V_i$ for each $i\in \{1, \ldots, k\}$.
\end{proof}

\section{Main theorem}\label{sec:mainthm}

We use the following theorem.

\begin{theorem}[Chudnovsky, Scott, and Seymour~\cite{ChudnovskySS2016}]\label{thm:longcycle}
 For every integer $n\ge 3$, the class of graphs having no induced cycle of length at least $n$ is $\chi$-bounded.
\end{theorem}

\begin{theorem}\label{thm:main}
 For every integer $n\ge 3$, the class of graphs with no $W_n$ vertex-minor is $\chi$-bounded. 
\end{theorem}
\begin{proof}
We recall that $R_1, R_2$ are the functions defined in Proposition~\ref{prop:recramsey}, and 
$M$ is the function defined in Proposition~\ref{prop:finalconfig2}.
Let $g_k$ be the $\chi$-bounding function of
Theorem~\ref{thm:longcycle} such that for every graph $G$ having no
induced cycle of length at least $k$ and all induced subgraphs $H$ of
$G$, $\chi(H)\le g_k(\omega(H))$.

Let $G$ be a graph such that
$\omega(G)\le q$ for some positive integer $q$ and it has no vertex-minor isomorphic to $W_n$.
Let $R_1:=R_1(2, M(n), q+1)$, $R_2:=R_2(2, M(n), q+1)$, and $r:=R_2n^{R_1}$.
We claim that $\chi(G) \le g_r(q)\cdot 2^{R_1}$.
Suppose not.
We may assume that $G$ is connected as we can color each connected component separately.

We will find a simple $(2, M(n))$-patched cycle with additional vertex
sets described in Proposition~\ref{prop:finalconfig2}. 

Let $v_1$ be a vertex of $G$ and for $i\ge 0$, let $L^1_i$ be the set of all vertices of $G$ whose distance to $v_1$ is $i$ in $G$. 
If each $L^1_j$ is $g_r(q)\cdot 2^{R_1-1}$-colorable, then $G$ is $g_r(q)\cdot 2^{R_1}$-colorable.  
Therefore there exists a level $L^1_t$ such that  $\chi(G[L^1_t])  >
g_r(q)\cdot 2^{R_1-1}\ge g_r(q)$.
Thus $G[L^1_t]$ contains an induced cycle of length at least $r$ by Theorem~\ref{thm:longcycle}.
Since $r\ge n+3$ and $G$ has no vertex-minor isomorphic to $W_n$, by Lemma~\ref{lem:fromlargerwheel}, we have $t\ge 2$.
Let $X_1:=L^1_t$,  $Y_1:=L^1_{t-1}$,  $Z_1:=L^1_0\cup L^1_1\cup \cdots \cup L^1_{t-2}$, and 
$r_1$ be the vertex in $L^1_0$.
We note that
for every $v\in N_G(Y_1)\cap Z_1$, there is a path $P=p_0p_1p_2 \cdots p_{t-2}$ where $p_0=r_1$, $p_{t-2}=v$ and for each $i\in \{0, \ldots, t-2\}$, $p_i\in L^1_{i}$. This path satisfies that $N_G(Y_1)\cap V(P)=\{v\}$.

Let $i$ be the maximum integer in $\{1, 2, \ldots, R_1\}$ such that
there exist disjoint vertex sets $X_i$ and $Y_1, \ldots, Y_i$ and $Z_1, \ldots, Z_i$ such that   
\begin{itemize}
\item  $\chi(G[X_i])>g_r(q)\cdot 2^{R_1-i}$,
\item for each vertex $v\in X_i$ and each $x\in \{1, \ldots, i\}$, $v$ has a neighbor in $Y_x$ and no neighbors in $Z_x$, 
\item for each $x\in \{1, \ldots, i\}$,  every vertex in $Y_x$ has a neighbor in $Z_x$, 
\item for each $x\in \{1, \ldots, i\}$, 
there exists a vertex $r_x\in Z_x$ where 
for every $v\in N_G(Y_x)\cap Z_x$, there is a path $P$ from $v$ to $r_x$ in $G[Z_x]$ with $N_G(Y_x)\cap V(P)=\{v\}$,
\item for distinct integers $x,y\in \{1, \ldots, i\}$ with $x<y$, there are no edges between $Z_x$ and $Y_y\cup Z_y$. 
\end{itemize}
Such $i$ exists, because $(X_1, Y_1, Z_1)$ satisfies these conditions. We claim that $i=R_1$.

Suppose that $i<R_1$. We choose a connected component $H$ of $G[X_i]$ with chromatic number more than $g_r(q)\cdot 2^{R_1-i}$ and let $v$ be a vertex in $H$.
For $j\ge 0$, let $L_j$ be the set of all vertices of $H$ whose distance to $v$ is $j$ in $H$. 
Since $H$ cannot be colored with $g_r(q)\cdot 2^{R_1-i}$ colors, 
there exists $t>0$ such that $\chi(H[L_t])> g_r(q)\cdot 2^{R_1-(i+1)}\ge g_r(q)$.
Since $H[L_t]$ has chromatic number at least $g_r(q)$, by Theorem~\ref{thm:longcycle}, it contains an
induced cycle of length at least $r$. 
Since $r\ge n+3$, by Lemma~\ref{lem:fromlargerwheel}, we have $t\ge 2$.
Let $X_{i+1}:=L_t$, $Y_{i+1}:=L_{t-1}$, $Z_{i+1}:=L_0\cup L_1\cup \cdots \cup L_{t-2}$, and let $r_{i+1}$ be the vertex in $L_0$.
Then $X_{i+1}$ and $Y_1, \ldots, Y_{i+1}$ and $Z_1, \ldots, Z_{i+1}$ satisfy these conditions, and it contradicts to the choice of $i$. 
Therefore we have $i=R_1$. 

Since $\chi(G[X_{R_1 }])>g_r(q)$, $G[X_{R_1 }]$ contains an induced cycle $q_1q_2 \cdots q_mq_1$ with $m\ge r$.
We apply Proposition~\ref{prop:intermediate} to the subgraph of $G$ induced on $\{q_1, q_2, \ldots, q_m\}\cup Y_1\cup \cdots \cup Y_{R_1}$.
Since $m\ge r=R_2 n^{R_1}$ vertices, 
by Proposition~\ref{prop:intermediate},
$G$ contains an $(R_1, R_2)$-patched cycle $(q_1q_2 \cdots q_mq_1, S_1, \ldots, S_{R_1})$ such that
for each $j\in \{1, \ldots, R_1  \}$, $S_j\subseteq Y_j$.
Furthermore, since $\omega(G)\le q$, by Proposition~\ref{prop:intermediatesimple}, 
$G$ contains a simple $(2, M(n))$-patched cycle $(q_1q_2 \cdots q_mq_1, S_a', S_b')$ such that 
$S_a'\subseteq S_a$ and $S_b'\subseteq S_b$ for some $a$ and $b$ with $1\le a<b\le
R_1$.
Note that 
\begin{itemize}
\item there are no edges between $\{q_1, q_2, \ldots, q_m\}$ and $Z_a\cup Z_b$, 
\item there are no edges between $Z_b$ and $S_a'$,
\item for each $x\in \{a,b\}$, every vertex of $S_x'$ has a neighbor in $Z_x$,
\item for each $x\in \{a,b\}$ and each vertex $v\in N_G(S_x')\cap Z_x$, there is a path $P$ from $v$ to $r_x$ in $G[Z_x]$ with $N_G(S_x')\cap V(P)=\{v\}$.
\end{itemize}
Therefore, by Proposition~\ref{prop:finalconfig2}, $G$ contains a vertex-minor isomorphic to $W_n$, which is contradiction.
\end{proof}

\end{document}